\documentclass[]{article}

\usepackage{amssymb,latexsym,amsmath,graphicx,color,amssymb,amsthm}     % Standard packages
\usepackage{subcaption}
\usepackage{mathtools}
\usepackage{enumitem}
\usepackage{titlesec}
\usepackage{diagbox}
\usepackage{ulem} 

\graphicspath{{fig/}}

\theoremstyle{plain}
\newtheorem{theorem}{Theorem}[section]
\newtheorem{lemma}[theorem]{Lemma}

 \newtheorem{assumption}{Assumption}

\theoremstyle{definition}

\theoremstyle{remark}

\DeclareMathOperator*{\argmin}{\text{argmin}}

\begin{document}

\title{Reduced-order Deep Learning for Flow Dynamics}
\author{Siu Wun Cheung, Min Wang, Wing Tat Leung , Eric T. Chung, Mary F. Wheeler}

\author{
	Min Wang\thanks{Department of Mathematics, Texas A\&M University, College Station, TX 77843, USA (\texttt{wangmin@math.tamu.edu})}
	\and
	Siu Wun Cheung\thanks{Department of Mathematics, Texas A\&M University, College Station, TX 77843, USA (\texttt{tonycsw2905@math.tamu.edu})}
		\and
	Wing Tat Leung\thanks{Institute for Computational Engineering and Sciences, The University of Texas at Austin, Austin, Texas, USA (\texttt{wleungo@ices.utexas.edu})}
	\and
	Eric T. Chung\thanks{Department of Mathematics, The Chinese University of Hong Kong, Shatin, New Territories, Hong Kong SAR, China (\texttt{tschung@math.cuhk.edu.hk})}
	\and
	Yalchin Efendiev\thanks{Department of Mathematics \& Institute for Scientific Computation (ISC), Texas A\&M University,
		College Station, Texas, USA (\texttt{efendiev@math.tamu.edu})}
	\and
	Mary Wheeler\thanks{Institute for Computational Engineering and Sciences, The University of Texas at Austin, Austin, Texas, USA (\texttt{mfw@ices.utexas.edu})}
}

\maketitle
\abstract{
In this paper, we investigate 
neural networks applied to multiscale simulations and discuss a design of
a novel deep neural network 
model reduction approach for multiscale problems.  
Due to the multiscale nature of the medium, 
the fine-grid resolution gives rise to a huge number of degrees of freedom.
In practice, 
low-order models are derived to reduce the computational cost. 
In our paper, we use 
a non-local multicontinuum (NLMC) approach, which represents the solution
on a coarse grid \cite{NLMC}.
Using multi-layer learning techniques, we formulate and learn
 input-output maps constructed  with NLMC on a coarse grid.
We study the features of the coarse-grid solutions that neural networks
capture via relating the input-output 
optimization to $\ell_1$ minimization of PDE
solutions. 
In proposed multi-layer networks, we can learn the forward operators
in a reduced way without computing them as in POD like approaches.
%In our paper, we use multi-layer network 
%for combined time stepping and reduced-order modeling, 
%where at each time step the appropriate important modes are selected.
We present
 soft thresholding operators as activation function, which
our studies show to have some advantages. With these activation functions,
the neural network identifies and selects important multiscale features 
which are crucial in modeling the underlying flow. 
Using trained neural network approximation of the input-output map,
we construct a reduced-order model for the solution approximation.
We use multi-layer networks  for the time stepping and reduced-order 
modeling, where at each time step the appropriate important modes are selected.
For a class of nonlinear problems, we suggest an efficient strategy.
Numerical examples are presented to examine the performance 
of our method. 
}

\section{Introduction}\label{sec:introduction}

%\subsection{Multiscale problems and model reduction}

Modeling of multiscale process has been of great interest in 
diverse applied fields. These include flow in porous media, mechanics, 
biological applications, and 
so on. However, resolving the fine-scale features of such processes could 
be computationally expensive due to scale disparity.
Researchers have been working on developing model reduction methods 
to bridge the information between the scales, in particular,
bring the fine-grid information to coarse grid.  
This generally reduces to constructing appropriate 
approximation spaces when 
solving a governing PDEs numerically.

Local model reduction is an important tool in designing reduced-order models
for practical applications.  Many approaches have been proposed
to perform model reduction, which include 
 \cite{ab05, egw10,  arbogast02, GMsFEM13, AdaptiveGMsFEM, brown2014multiscale, ElasticGMsFEM, ee03, abdul_yun, ohl12, fish2004space, fish2008mathematical, oz07, matache2002two, henning2009heterogeneous, OnlineStokes, chung2017DGstokes,WaveGMsFEM, MsDG}, 
where the coarse-grid equations are formed and the
parameters are computed or found via inverse problems 
\cite{inverse_Zabaras, made08, IB2013, memd10, deepconv_Zabaras}. 
Adopting  upscaling ideas, methods such as Generalized Multiscale Finite Element Method (GMsFEM) \cite{efendiev2013generalized,chung2016adaptive}, Non-local Multicontinuum Methods (NLMC) \cite{chung2018non} have been developed. They share similar ideas in constructing the approximation space, where local 
problems are solved 
to obtain the basis functions. The numerical solution is then searched within the space that constituted by such basis. With extra local computation, these multiscale spaces can capture the solution using  small degrees of freedom. 
These methods have been proved to be
 successful in a number of multiscale problems. 
Global model reduction
methods seeks a reduced dimensional solution approximation 
using global basis functions. When seeking numerical solutions within 
a high-dimensional approximation space $V_h$, the solutions could 
be sparse.  Proper Orthogonal Decomposition (POD)
\cite{bourlard1988auto} is often used to find a low-dimensional 
approximate representation of the solutions.  
%This method is 
%effective when the solution
%does not span over entire $V_h$ in many cases. 

To exploit the advantages of local and global 
model reduction techniques, 
global-local model reduction approaches have been explored
in the past (see, e.g.,
\cite{nonlinear_AM2015,GhommemJCP2013,EGG_MultiscaleMOR,yang2016fast}).
In global-local model reduction methods, the input-output
map is approximated using POD or other approaches in combination
with multiscale methods. Multiscale methods allow adaptive
computing and re-computing global modes \cite{yang2016fast}.

%Artificial Neural Network was first proposed in the 1940s and, it has recently
%gained popularity. It has shown its supreme ability in reproducing and modeling nonlinear processes in areas such as computer vision, speech recognition, etc. Mathematicians have also attempted to expand its application to problems that requires model reduction. The paper \cite{wang2018prediction} used feed forward neural network to build a local reduced-order multiscale model, while other investigations on neural network have shown that it is strongly related to Principle Component Analysis (PCA/POD) \cite{bourlard1988auto}, which is commonly used for global model reduction.

In this paper, we investigate learning strategies for multiscale methods
by approximating the input-output map using multi-layer neural networks.
Nowadays, neural networks have been commonly used in many applications
\cite{schmidhuber2015deep}. In our paper, we use neural networks in conjunction with
local reduced-order models.
Our studies show that the neural networks provide some type of 
reduced-order global modes when considering linear problems.
%\marginpar{Min: I'm not sure what "spectral modes " refers to here.}
In the paper, we investigate these global modes and their relation
to spectral modes of input-output map. 
With a local-global model reduction technique, we 
limit our search of solution to a smaller zone (see Figure 
\ref{fig:illu_sparse_solution} for an illustration) that fits the 
solution populations better. 
Our studies show that the multi-layer network provides a good
approximation of coarse-grid 
input-output map and can be used in conjunction with
observation data.
Moreover, in multi-layer networks, we can learn the forward operators
in a reduced way without computing them as in POD like approaches.
In our paper, we use multi-layer network 
for combined time stepping and reduced-order modeling, 
where at each time step the appropriate important modes are selected.
%We remark that we can use observed data to learn multiscale model as in \cite{DNN_NLMC}. 
 A soft thresholding neural network 
is suggested based on our analysis.

%\marginpar{do we need this figure?}
\begin{figure}[htbp]
	\centering
	\includegraphics[width=0.4\textwidth]{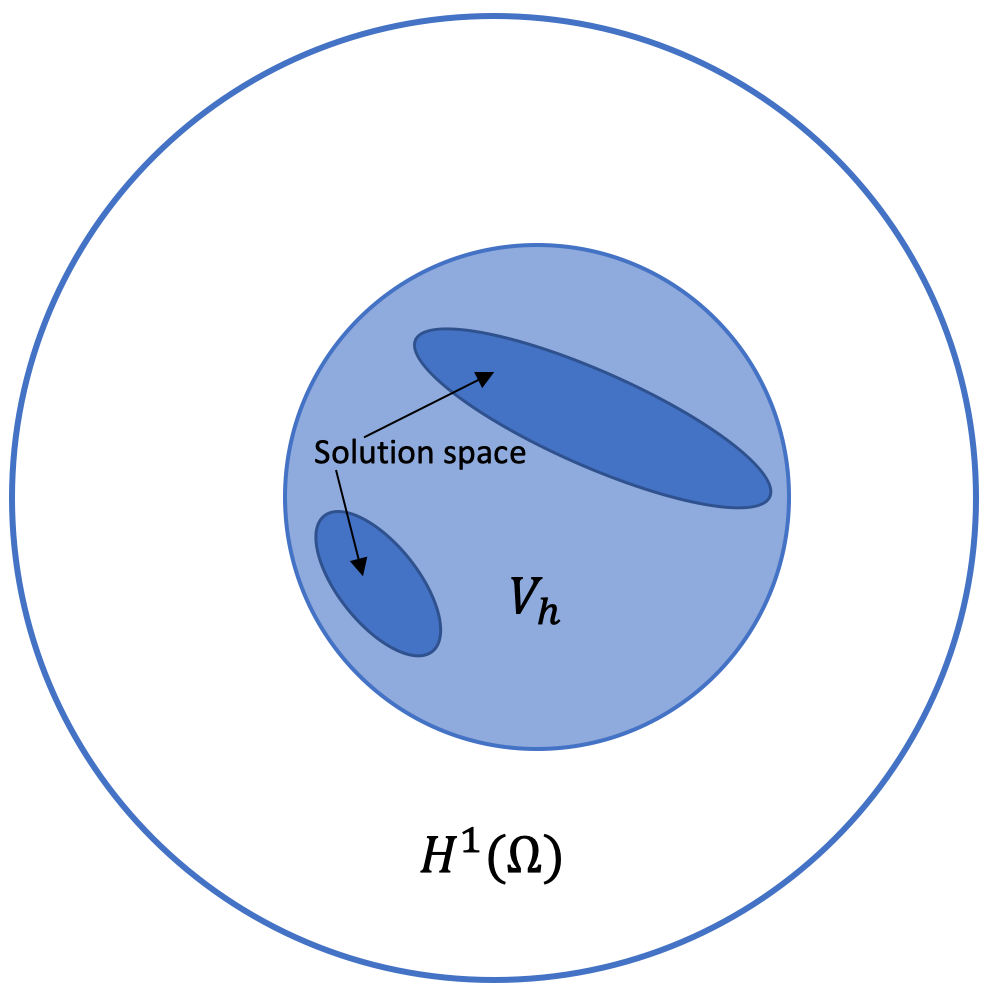}
	\caption{Illustration of sparsity of our solution.}
	\label{fig:illu_sparse_solution}
\end{figure}

Our previous work on neural networks and multiscale methods
\cite{DNN_NLMC} have mostly focused on incorporating the observed
data. 
%With previous studies of neural network in 
%learning multiscale models and conducting model reduction \cite{xxx}, 
%we  propose a reduced-order neural network aiming at 
%modeling a nonlinear flow dynamics with observed coarse-scale data. 
%The network will also work as an model reduction scheme.
In this paper, 
the key tasks
 are: (1) reproducing the coarse-scale multiscale model; 
(2) simultaneously determining the sparse representation 
of solution; and, 
(3) identifying the dominant multiscale modes and provide 
a new basis system which can be used to reduce the multiscale model dimension.
There are a few advantages of using neural networks to accomplish 
these tasks. For (1) \& (2), the computation is inexpensive and robust. 
For (3), the dominant multiscale modes found by the neural network  can be 
generalized to conduct model reduction solutions.
Our studies suggest the use of soft
thresholding and we observe a relation between the soft thresholding and 
$\ell_1$ minimization.

To handle the nonlinearities in PDEs, we propose the use of clustering.
Clustering is an effective pretreatment for learning and prediction. 
Specifically, the solution of nonlinear PDEs clusters around
some nonlinear values of the solution
(see Figure \ref{fig:illu_sparse_solution} for an illustration).
 Sometimes, solution clusters can be
 predicted without computing solutions. For example, if we are interested 
in the solutions to the flow equation
\begin{equation}\label{eq:flow}
\dfrac{\partial u}{\partial t} -\text{div}(\kappa(t,x,u) \nabla u) = f,
\end{equation}
where $\kappa(t,x,u) =\hat{ \kappa}(x) u^2$, and $\hat{ \kappa}(x)$
has a limited number of configurations as shown in Figure \ref{fig:perm}, 
then we could anticipate the solutions to accumulate for 
each configuration of $\hat{ \kappa}(x)$. This is an important example
of channelized permeability fields that arise in many subsurface 
applications.
Thus, if we can use this prejudged information and divide the training samples 
into clusters, we are likely to get an accurate model with significantly 
less training resources.

\begin{figure}	
	\centering
	\includegraphics[width=0.8\textwidth]{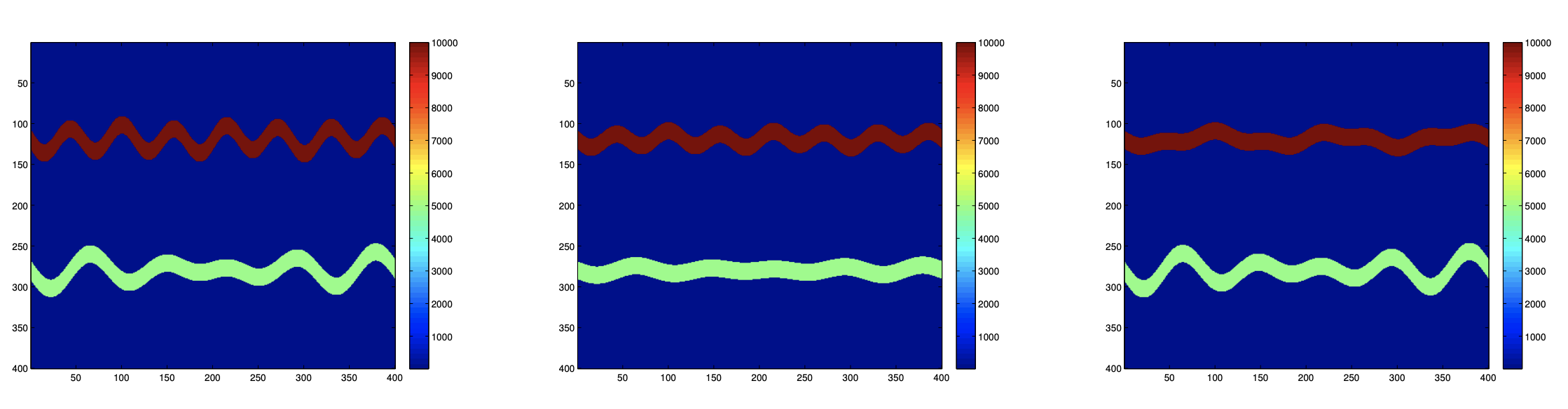}
	\caption{Illustration of different configurations of $\hat{\kappa}(x).$}
	\label{fig:perm}
\end{figure}

%One main objective of this paper is to derive a more robust network
%and provide a better understanding of the application 
%of neural networks to multiscale simulations. 
%In particular, our goal is to 
%investigate how neural network captures the important modes 
%relevant to multiscale features of the solution, 
%by establishing an appropriate relation between $\ell_1$ minimization and model reduction 
%using a soft thresholding nonlinearity. 
%Multi-layer network is used for combined time stepping and reduced-order modeling, 
%where at each time step the appropriate important modes are selected.
%We also suggest an efficient strategy for some class of nonlinear problems 
%that arise in porous media applications.

The main contributions of the paper are the following.
(1) We study how neural network captures the important
 multiscale modes related to the features of the solution.
(2) We relate $\ell_1$ minimization to model reduction and derive a more robust network for our problems using a soft thresholding.
(3) We suggest an efficient strategy for some class of nonlinear problems that arise in porous media applications
(4)  We use multi-layer networks  for combined time stepping and reduced-order modeling, where at each time step the appropriate important modes are selected.
We remark that we can use observed data to learn multiscale model as in \cite{DNN_NLMC}.

The paper is organized as follows. Section \ref{section_pre} will be a preliminary introduction of the multiscale problem with a short review on multiscale methods. A description for general neural network is also provided.  Section \ref{section_nn} mainly focuses on discussions on the reduced-order neural network. The structure of the neural network is presented. Section\ref{sec:discuss} later discuss the proposed neural network from different aspects and propose a way to conduct model reduction with the neural network coefficient. We also present the relation between the soft thresholding neural network and a $\ell_1$ minimization problem in this section. Section \ref{section_example} provides various numerical examples to verify the predictive power of our proposed neural network and provide support to the claims in Section \ref{sec:discuss}. Lastly, the paper is concluded with Section \ref{section_conclude}.

\section{Preliminaries}\label{section_pre}
\subsection{Governing equations and local model reduction }\label{sec:NLMC}

We consider a nonlinear flow equation
\begin{equation}
\dfrac{\partial u}{\partial t} -\text{div}(\kappa(t,x,u) \nabla u) = f,\tag{ \ref{eq:flow}}
\end{equation}
in the domain $(0,T) \times \Omega$. Here, $\Omega$ is the spatial domain, 
$\kappa$ is the permeability coefficient which can have a multiscale 
dependence with respect to space and time, and $f$ is a source function. 
The weak formulation of \eqref{eq:flow} involves finding $u\in H^1(\Omega)$ such that
\begin{equation}\label{weak}
\left(\frac{\partial u}{\partial t}, v\right) + (\kappa(t,x,u) \nabla u, \nabla v ) = (f, v), \quad \forall v \in H^1(\Omega).
\end{equation}
If we  numerically  solve this problem in a $m$-dimensional approximation space $V_h \subset H^1(\Omega)$, and use an Euler temporal discretization, the numerical solution can be written  as
\begin{equation}\label{eq:sum_solution}
u_h(t_n, x)= u_h^n = \sum_{j=1}^m \alpha_j^n \phi_j(x),
\end{equation}
where $\{\phi_j\}_{j=1}^m$ is a set of basis for $V_h$. Moreover, the problem \eqref{weak} can then be reformulated as: find $u_h \in V_h$ such that
\begin{equation}\label{eq: numerical_PDE}
\left(\frac{ u_h^{n+1} - u_h^{n} }{\Delta t}, v_h\right) + (\kappa(t_n, x, u_h^{\nu})\nabla u_h, \nabla v_h ) = (f, v_h), \quad \forall v_h \in V_h,
\end{equation}
where $\nu = n$ or $n+1$ corresponds to explicit and implicit Euler discretization respectively. Here, $(\cdot, \cdot)$ denotes the $ L_2$ inner product.\\

For problems with multiple scales, a multiscale basis function for each coarse node is computed following the idea of upscaling, i.e., the problem can be solved with a local model reduction. Instead of using the classic piece-wise polynomials as the basis functions, we construct the local multiscale basis following NLMC \cite{DNN_NLMC} and use the span of all such basis function as the approximation space $V_h$. More specifically, for a fractured media(Figure \ref{fig:fine_coarse_mesh}), the basis functions of \eqref{eq:flow} can be constructed by the following steps:
\begin{itemize}
	\item Step 1: Partition of domain
	\begin{equation}
	\Omega = \Omega_m \bigoplus_n d_n \Omega_{f,n}.
	\end{equation}
Assume $\mathcal{T}^H$
	is a coarse-grid partition of the domain $\Omega$ with a mesh size $H$ which is further refined into a fine mesh $\mathcal{T}^h$. Denote by $\{K_i | \quad i = 1, \cdots, N\}$ the set of coarse elements in $\mathcal{T}^H$, where $N$ is the number of coarse blocks. We also define the oversampled region $K^+_i$ for each $K_i$, with a few layers of neighboring coarse blocks, see Figure \ref{fig:fine_coarse_mesh} for the illustration of $K_i$ and $K_i^+$.
	We further define the set of all fracture segments within a coarse block $K_j$ as $F^{j} = \{f_n^{(j)}| 1\leq n\leq L_j \} =\{\cup_{n} \Omega_{f,n}\}\cap K_j$ and $L_j = dim\{F^{j}\}$.
	\begin{figure}[!htbp]
		\centering
		\includegraphics[width=0.5\textwidth]{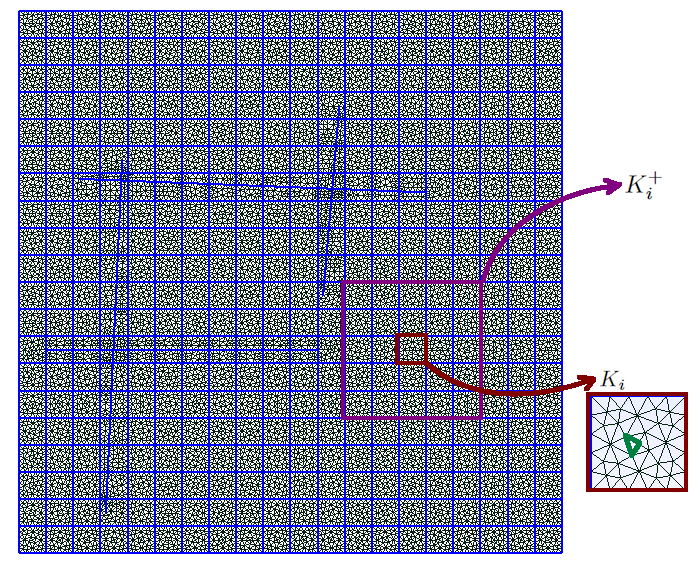}
		\caption{Illustration of coarse and fine meshes. }
		\label{fig:fine_coarse_mesh}
	\end{figure}

	\item Step 2: Computation of local basis function in $K_i^+$.

	The basis for each over-sampled region $\psi_m^{(i)}$ solves the following local
	constraint minimizing problem on the fine grid
	\begin{equation}\label{eq:basis}
	\begin{aligned}
	& a(\psi_m^{(i)}, v) + \sum_{K_j \subset K_i^+} \left(\mu_0^{(j)} \int_{K_j}v + \sum_{1 \leq m \leq L_j} \mu_m^{(j)} \int_{f_m^{(j)}} v \right) = 0, \quad \forall v\in V_0(K_i^+), \\
	& \int_{K_j}\psi_m^{(i)} = \delta_{ij} \delta_{0m}, \quad \forall K_j \subset K_i^+, \\
	& \int_{f_n^{(j)}} \psi_m^{(i)} = \delta_{ij} \delta_{nm}, \quad \forall f_n^{(j)} \in F^{(j)}, \; \forall K_j \subset K_i^+,
	\end{aligned}
	\end{equation}
where $ a(u, v) = \int_{D_m } \kappa_m  \nabla u \cdot \nabla v + \sum_i \int_{D_{f,i}} \kappa_i  \nabla_f u \cdot \nabla_f v$, $\mu_0^{(j)}, \mu_m^{(j)}$ are Lagrange multipliers while $V_0(K_i^+) = \{v\in V(K^+_i)|v= 0 \text{ on } \partial K_i^{+}\}$ and $V(K_i^+)$ is the fine grid space over an over-sampled region $K^{+}_i$. By this way of construction, the average of the basis $\psi_0^{(i)}$ equals $1$ in the matrix part of coarse element $K_i$, and equals $0$ in other parts of the coarse blocks $K_j \subset K_i^+$ as well as any fracture inside $K_i^+$. As for $ \psi_m^{(i)},m\geq1 $, it has average $1$ on the $m$-th fracture continua inside the coarse element $K_i$, and average $0$ in other fracture continua as well as the matrix continua of any coarse block $K_j \subset K_i^+$. It indicates that the basis functions separate the matrix and fractures, and each basis represents a continuum.
\end{itemize}
The resulting multiscale basis space is finally written as $V_h = span \{\psi_m^{(i)}| 0\leq m\leq L_i, 1\leq i\leq N\}$. By renumbering the basis, we denote all basis function as $\{\phi_j\}_{j=1}^{m}$, where $m = \sum_{i=1}^{N} (L_j+1)$.

Thus, the coefficient $U^n = (\alpha_1^n, \alpha_2^n, \ldots, \alpha_m^n)^T$ satisfies the recurrence relation 
%\marginpar{(8) is not true for explicit scheme}
\begin{equation}\label{eq:recurrent_relation}
U^{n+1} = (M + \Delta t A^{\nu})^{-1} (M U^n + \Delta t F^{n}),
\end{equation}
where $M$ and $A^{\nu}$ are the mass matrix and the stiffness matrix
with respect to the basis $\{\phi_j\}_{j=1}^m$, $\nu$ can be $n$ or $n+1$ depending on the temporal scheme that we use. We have
\[ \begin{split}
[M]_{ij} & = \int_\Omega \phi_i(x) \phi_j(x) \, dx ,\\
[A^{\nu}]_{ij} & = \int_\Omega \kappa(t_n,x,u^\nu_h) \nabla \phi_i(x) \cdot \nabla \phi_j(x) \, dx.
\end{split} \]
%Here $U^n \mapsto (M + \Delta t A^{n})^{-1}$ is a nonlinear map.

We claim that a global model reductions can be conducted to the problem described above, as solution $u_h^n(x)$ in many cases can be sparse in $V_h$ even if $V_h$ is a reduced-order space. For instance, $u_h(t_n,x)$ is strongly bonded to initial condition $u_h(t_0,x)$. It can be foreseen that if initial conditions are chosen from a small subspace of $V_h(\Omega)$, $u_h^n(x)=u_h(t_n,x)$ 
is also likely to accumulate somewhere in $V_h$. In other words, the distribution of coefficients $U^{n}$ can hardly expand over the entire $\mathbb{R}^m$ space but only lies in a far smaller subspace.

Other physical restrictions to the problem could also narrow down the space of solution. This indicates that $u_h(t_n,x)$ can be closely approximated with less degrees of freedom. Section \ref{section_nn} and Section \ref{sec:discuss} will be discussing how to identify dominant modes in the space of $U^n$ using a neural network.

\subsection{Neural Network} \label{neural_net_section}
A neural network is usually used to learn a map from existing data. For example,
 if we are given samples $ \{(x_i, y_i)\}_{i=1}^{L}$ from the map $\mathcal{F}: X \to Y$, i.e., $\mathcal{F}(x_i) = y_i$ for $1\leq i \leq L$, and we would like to predict the value of $\mathcal{F}(x_i)$ for $i = L+1, \cdots, L+M$. With the help of neural network, this problem can be reformulated as the optimization problem. The optimization takes $ \{(x_i, y_i)\}_{i=1}^{L}$ as training samples and produces a proper network coefficient $\theta^*$ starting from some initialization chosen at random. More specifically, if the neural network $\mathcal{NN}(\cdot)$ has a feed-forward fully-connected structure, then
 \begin{equation}
\mathcal{NN}(x; \theta) = W_n\sigma(\cdots\sigma(W_2\sigma(W_1x+ b_1 )+b_2)\cdots) +b_n.
 \end{equation}
 Here, $\theta$ represents all tuneable coefficients in neural network $\mathcal{NN}(\cdot)$ and $\sigma(\cdot)$ is some nonlinear activation function. There are many choices of such nonlinear functions \cite{ramachandran2018searching}, while the most common ones used are ReLU and $\tanh$.

 For the network defined above, $\theta$ is indeed defined as $\theta = \{W_1, b_1,W_2, b_2,\cdots, W_n, b_n\}$. We will use the output $\mathcal{NN}(x_i)$ to approximate the desired output $y_i$. The difference between them will be measured using a cost function $\mathcal{C}(\cdot)$.
For example, we can take the mean square error
and the loss function is defined as
\begin{equation}\label{eq: loss_function}
\mathcal{C}(\theta ) = \frac{1}{N}\sum_{i= 1}^{N} (y_i-\mathcal{NN}(x_i))^2,
\end{equation}
which measures the average squared difference between the estimated values and the observed values. The neural network is then optimized by seeking $\theta^*$ to minimize the loss function.
\begin{equation}
\theta^* = \argmin_\theta \mathcal{C}(\theta).
\end{equation}
Numerically, this optimization problem can be solved with a stochastic gradient descent(SGD) type method \cite{kingma2014adam}. By calculating the gradient of the loss function, the coefficient $\theta$ is iteratively updated in an attempt to minimize the loss function. This process is also referred as ``training''.
Once the loss is minimized to a satisfactory small level, the neural network parameters $\theta^*$ is decided, and further, the overall neural network architecture \eqref{eq:w2_multi_layer} is constructed.  
The predictions can then be given as
$\mathcal{NN}(x_i;\theta^*)$ for $ i = L+1, \cdots, L+M$.

\section{Reduced-order neural network}\label{section_nn}

In this section, we present the construction of reduced-order neural network. 
We propose a reduced-order neural network that can model a time series. 
Moreover, if there exists a basis 
that approximates the solution for each
time step with sparse coefficients, then the proposed 
neural network can identify such basis from the training samples.  
Specifically, Subsection \ref{sec: macro_nn} will discuss the macroscopic 
structure of the proposed neural network while Subsection \ref{sec:subnet} 
will discuss two designs of the sub-network of the network 
with more details. Subsection \ref{sec:full_nn} later assembles 
the full multi-layer neural network. 

\subsection{Reduced-order neural network structure}\label{sec: macro_nn}

We propose a reduced-order neural network as shown in Figure \ref{fig:nn_full}. Here, the full network is constituted by several sub-networks. 
Each sub-network $\mathcal{N}^n(\cdot)$ is expected to model a 
one-step temporal evolution from $\mathbf{x}^{n}$ to $\mathbf{x}^{n+1}$ 
in a time series 
$\vec{\mathbf{x}} = [\mathbf{x}^0,\mathbf{x}^1, \cdots, \mathbf{x}^n]$.

Sub-networks should have a general structure as shown in 
Figure \ref{fig:nn_block}. The specific design will vary 
depending on the problem we are modeling 
(see discussions in Subsection \ref{sec:subnet}). The sub-network is built 
in a way that the input $\mathbf{x}^n$ will be first feed 
into a multi-layer fully-connected network named as ``operation layer''.  
This layer is intended for the neural network to capture the map 
between two consecutive time steps. The output of the operation layer 
is then fed into
a soft thresholding function to impose sparsity to the solution 
coefficient.  Lastly, the data will be processed with a 
``basis transform layer'' in which a new basis set will be learned. 
With the basis set, one can represent the solution with sparse 
coefficients assuming such representation exists.

%We also remark that activation functions used in the operation layer can be any commonly used ones such as ReLU and its variants, but only soft thresholding function should be applied after the data has been output by it, which is indeed for the sake of sparsity of the solution.

\begin{figure}[htbp]
	\centering
	\begin{subfigure}[b]{0.85\textwidth}
		\includegraphics[width=\textwidth]{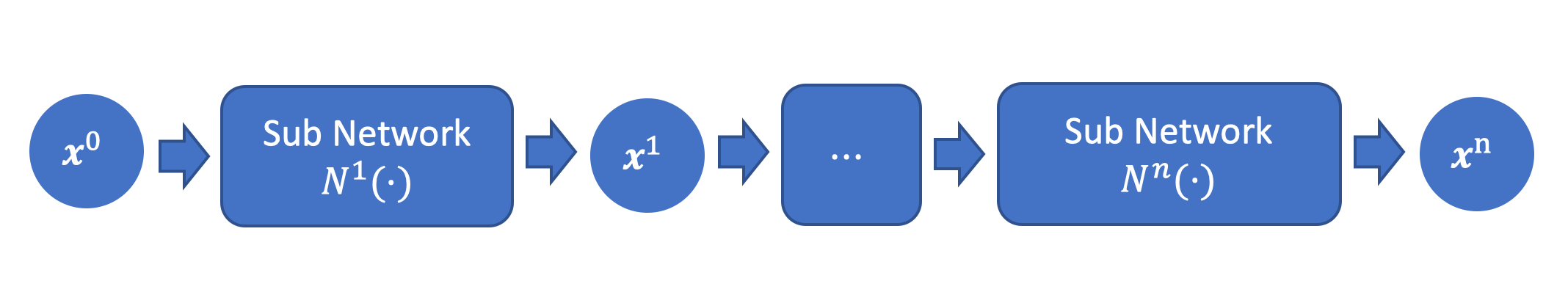}
		\caption{Multi-layer  reduced-order  neural network $\mathcal{NN}(\cdot)$}
		\label{fig:nn_full}
	\end{subfigure}
	~
	\begin{subfigure}[b]{0.85\textwidth}
		\includegraphics[width=\textwidth]{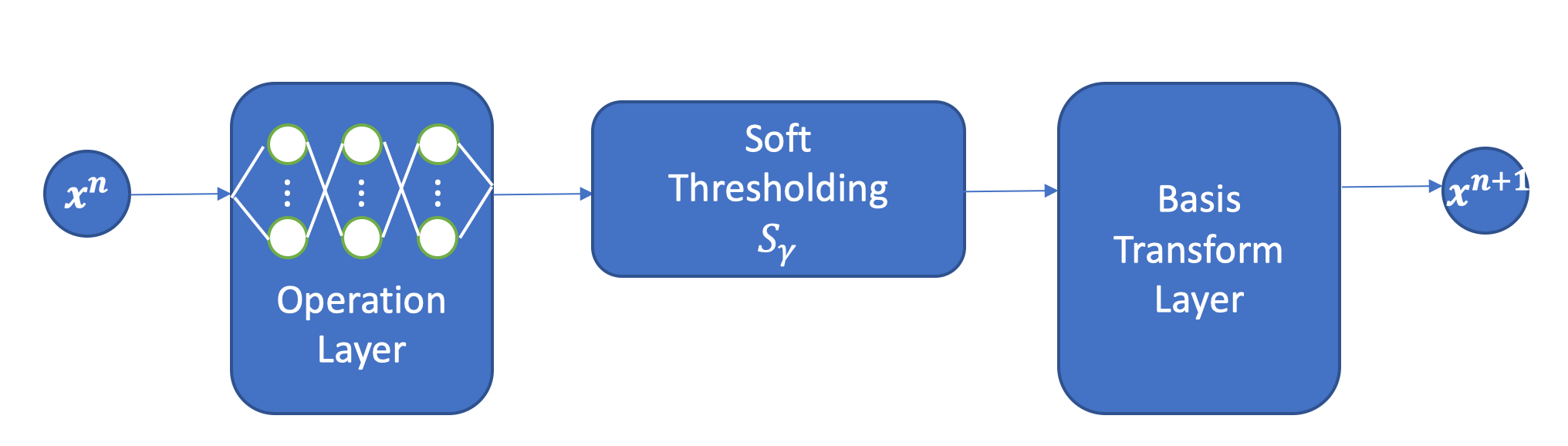}
		\caption{Sub-network $\mathcal{N}^n(\cdot)$}
		\label{fig:nn_block}
	\end{subfigure}
	%	~
	\caption{Reduced-order neural network structure.}
	\label{fig:nn}
\end{figure}

\subsection{Sub-network}\label{sec:subnet}
%\marginpar{Tat: Section 3 add more motivation also cite section 4 for details.}
In this subsection, we present two designs of the sub-network $\mathcal{N}^{n}(\cdot)$. One is for modeling linear dynamics while the other is designed for nonlinear dynamics. One can choose from these two options when assembling the full network depending on the dynamics of interest. Both sub-network designs are intended to learn a new set of basis and then impose the sparsity to the solution coefficient in the new system while learning dynamics.

\subsubsection{Sub-network for linear process}
 
We first present the sub-network for modeling linear dynamics. 
It can be used to model the one-step flow described in \eqref{eq:flow}, 
where
we define the sub-network for the $n$-th time step as  
\begin{equation}
\label{eq:w2_nn}
\mathcal{N}^n(\mathbf{x}^n; \theta_n) := W_n^2S_\gamma(W_n^1\cdot \mathbf{x}^n + b_n).
\end{equation}  
Here, for the sub-network parameter $\theta_n = (W_n^1, W_n^2, b_n)$,
$W_n^1$ and $b_n$ are in the operation layer and 
$W_n^2$ works as the basis transformation layer. 
$S_\gamma$ is the soft thresholding function defined point-wise as 
$S_\gamma: \mathbb{R} \to \mathbb{R}$
\begin{equation}\label{eq:s_gamma}
S_\gamma(x) = \text{sign}(x) (\vert x \vert - \gamma)_+ \begin{cases}
x - \gamma & \text{ if } x \geq \gamma, \\
0 & \text{ if } -\gamma < x < \gamma, \\
x + \gamma & \text{ if } x \leq -\gamma.
\end{cases}
\end{equation}
We further require $ W_n^2$ to be an orthogonal matrix, i.e. $(W_n^2)^T\cdot (W_n^2)  = I $. To this end, 
we train the network \eqref{eq:w2_nn} with respect to the cost function 
\begin{equation}\label{cost}
\mathcal{C}^n(\theta_n) = \|\mathbf{x}^{n+1}-\mathcal{N}^n(\mathbf{x}^n; \theta_n) \|_{2}^2 + \eta_n \|(W_n^2)^T\cdot (W_n^2) - I \|_{1}
\end{equation}
with a penalty on the orthogonality constraint of $W_n^2$. $\eta$ is a hyper coefficient for adjusting the weight of the $\ell_1$ regularization term.  Here, $\mathbf{x}^{n}$ is the input of  $\mathcal{N}^n(\cdot; \theta_n)$, and $\mathcal{C}^n(\cdot)$ measures the mismatch between true solution $\mathbf{x}^{n+1}$ and the prediction $\mathcal{N}^n(\mathbf{x}^n)$ while forcing $W_n^2$ to be orthogonal. We remark that this cost functions is only a part of the full cost function that we will be discussed in Subsection \ref{sec:full_nn}. 

%\subsubsection{Sub-network with basis transformation}

With such a design, the trained neural sub-network $\mathcal{N}^n(\cdot; \theta_n^*)$ will be able to model the input-output map specified by the training data while producing a matrix $ W^2_n$ whose columns forms an orthogonal basis in $\mathbb{R}^m$.

%In other words, as discussed in the introduction, we expect the solution $U^{n+1}$ to accumulate in a subregion of  $\mathbb{R}^m$, and thus should can be approximated with only the orthogonal dominant mode of the $U^{n+1}$ population. We would like to express $U^{n+1}$ with these dominant modes to obtain a sparse representation. The columns of $W_n^2$ are exactly these dominant modes and therefore $ U^{n+1, W^2}$ is sparse.
%
% $W_n^2$ could also be considered as a new coordinate system. This true as $\mathcal{N}^n(\cdot)$ requiring that the columns of $W_n^2$ form a nearly orthogonal basis in $\mathbb{R}^m$ by training in respect to cost function \eqref{cost}. Thus, the output of $\mathcal{N}^n(U^n)$ is sparse with respect to $W^2_n$-system.

\subsubsection{Sub-network for nonlinear process }

Similar to the linear case, a sub-network for nonlinear process can also be designed.  When the output $\mathbf{x}^{n+1}$ is non-linearly dependent on input $\mathbf{x}^n$, we make use of a $d_n$-layer feed-forward neural network $ \tilde{\mathcal{N}}(\cdot) $ to approximate the input-output map. $ \tilde{\mathcal{N}}(\cdot) $ will work as the operation layer of the sub-network. The output of $ \tilde{\mathcal{N}}(\cdot) $ is then processed with soft-thresholding and a basis transformation layer $W_n^{2}$. We define the sub-network to be
\begin{equation}\label{eq:nonlinear_nn_w2}
\begin{split}
\mathcal{N}^n(\mathbf{x}^n; \theta_n) &:= W_n^2 S_\gamma(\tilde{\mathcal{N}} (\mathbf{x}^n; \theta_n)), \\
\tilde{\mathcal{N}}(\mathbf{x}^n; \theta_n) & :=  W_n^{1,d_n}(\cdots(\sigma (W_n^{1,3}\sigma (W_n^{1,2}\sigma(W_n^{1,1} \mathbf{x}^n+b_n^1)+b_n^2)\cdots+b_n^{d_n}),
\end{split}
\end{equation}
where $\theta_n = \left(W_n^{1,1}, W_n^{1,2}, \ldots, W_n^{1,d_n}, W_n^2, b_n^1, b_n^2, \ldots, b_n^{d_n}\right)$ 
is the sub-network parameter, and 
$\sigma(\cdot)$ is a nonlinear activation function.
The cost function for training the sub-network parameter is again defined in \eqref{cost}. 
We also remark that if $d_n=1$ and $\sigma=\mathbf{1}$ is the point-wise identify function, 
then the network structure \eqref{eq:nonlinear_nn_w2} is reduced to \eqref{eq:w2_nn}.

Another approach to reduce the difficulty of reproducing a nonlinear process is clustering. Instead of using a single network to approximate complicated nonlinear relations, we use different networks for different data cluster. As discussed in the Section \ref{sec:introduction}, the clusters of the solutions can be predicted.  Thus, separate the training samples by cluster can be an easy and effective way to accurately recover complicated process.  Discussions on clustering and numerical examples are presented in Section \ref{section_clustering}.

\subsection{Multi-layer reduced order network}\label{sec:full_nn}

Now we construct the full neural network $\mathcal{NN}(\cdot;\theta)$ by stacking up the sub-network $\mathcal{N}^n(\cdot;\theta_n)$. More precisely, $\mathcal{NN}(\cdot; \theta): \mathbb{R}^m \to \mathbb{R}^m$ is defined as:
\begin{equation}
\mathcal{NN}(\mathbf{x}^0; \theta) := \mathcal{N}^n(\cdots \mathcal{N}^1(\mathcal{N}^0(\mathbf{x}^0; \theta_0); \theta_1)\cdots; \theta_{n}),
\label{eq:w2_multi_layer}
\end{equation}
where $\mathcal{N}^n(\cdot;\theta_n)$ is defined as in \eqref{eq:w2_nn} or \eqref{eq:nonlinear_nn_w2} depending the linearity of the process, and $\theta = (\theta_0, \theta_1, \ldots, \theta_{n})$ is the full network parameter. We use such $\mathcal{NN}(\cdot; \theta)$ to approximate time series $\vec{\mathbf{x}} = [\mathbf{x}^0, \mathbf{x}^1, \cdots, \mathbf{x}^{n+1}]$. 
Denote the output of $(t+1)$-fold composition of sub-network $\mathcal{N}^t$ as
\begin{equation}
\mathbf{o}^{t+1} : = \mathcal{N}^t(\cdots \mathcal{N}^1(\mathcal{N}^0(\mathbf{x}^0; \theta_0);\theta_1) \cdots; \theta_{t}). 
\end{equation}
Then $\mathbf{o}^{t+1} $ works as a prediction of the solution at time $t+1$.
The full cost function is then defined as
\begin{equation}\label{eq:cost_multi}
\mathcal{C}(\theta) =\sum_{t= 0}^{n}\|\mathbf{o}^{t+1} -\mathbf{x}^{t+1} \|_{2}^2 + \eta_t \|(W_t^2)^T\cdot (W_t^2) - I \|_{1},
\end{equation}
where $\vec{\mathbf{x}}$ is the true time sequence, and $\eta_t$ is hyper-parameter stands for the weight of the regularizer while $\theta$ represents all tuneable parameters of $\mathcal{NN}(\cdot)$. Each layer (sub-network) corresponds to a one-step time evolution of the dynamics.

Suppose we have $L$ training samples, 
$$\vec{\mathbf{x}_i}= [(\mathbf{x}_i^0),( \mathbf{x}_i^1, \mathbf{x}_i^2,\cdots, \mathbf{x}_i^{n+1})],\quad\quad 1\leq i \leq L.$$
%and compute the corresponding forward dynamics matrix $(M + \Delta t A^{n+1}(U^n))^{-1}$.
%We define the loss function $\mathcal{L}^n$ by
%\begin{equation}
%\label{eq:loss}
%\mathcal{L}^n(\theta^n) = \sum_{i=1}^L \left\|
%(M + \Delta t A^{n+1}(U^n_i))^{-1} - \sum_{k=1}^s \mathcal{C}^n_k(U^n; \theta^n) \Phi_k^n \right\|_F^2,
%\end{equation}
%where $\| \cdot \|_F$ is the Frobenius norm on matrices.
%To train the network $\mathcal{N}^{n}$,
The optimal parameter $\theta^*$ of $\mathcal{NN}(\cdot)$ is then determined by optimizing the cost function $\mathcal{C}(\theta)$
%\begin{equation}
%\theta^* = \argmin_\theta \frac{1}{L} \sum_{i=1}^{L}\mathcal{C}(\vec{\mathbf{x}_i}; \theta), \tag{\ref{eq: loss_function}}
%\end{equation}
subject to this training set $\{\vec{\mathbf{x}_i}\}_{i=1}^{L}$ as discussed Section \ref{neural_net_section}. 
%$\theta $ represents the set $\{W_t^{1},W_t^{2}|\ 0\leq t\leq n\}$ for linear case and $\{W_t^{1,j},W_t^{2}|\ 1\leq j \leq d_t, 0\leq t\leq n\}$ for nonlinear case. 
Once $\theta^{*}$ is decided, predictions can be made for testing samples $[\mathbf{x}_i^1,\cdots, \mathbf{x}_i^{n+1}]$ by $\mathcal{NN}(\mathbf{x}_i^0), i > L.$

\section{Discussions and applications}\label{sec:discuss}
%\marginpar{Professor  Efendiev: Add steps such as 4.2(as sparsity assumptions ) -- 4.4(sparsity is related to $\ell_1$)--4.5(remark, soft thresholding extract modes)-4.1 (neural network can approximated the in-out map )--4.3 (algorithm, extract basis, reduce operator )-- clustering }
In this section, we discuss some theoretical aspects of
proposed neural networks. Specifically, we use the proposed network 
to model fluid dynamics in heterogeneous media, as described 
in \eqref{eq:flow}. 
First, we relate the soft thresholding network with 
$\ell_1$ optimization problem. 
Secondly, we explore how 
learned coefficients of neural network are 
related to the map that is being approximated. 
Thirdly, 
based on the understanding of proposed neural network, 
we present a way to utilize the trained network coefficient 
to construct a reduced-order model.
%\marginpar{Tat: describe the function of each  subsection}

Specifically, we consider a one-layer neural network for single-step linear dynamics
\begin{equation}\label{eq:one_layer}
\mathcal{NN}(\mathbf{x}; \theta) = \mathcal{N}(\mathbf{x}; \theta): = W^2S_\gamma(W^1 \mathbf{x} + b),
\end{equation}
omitting the indices for time step $n$. 
We train the neural network $\mathcal{NN}(\cdot;\theta)$ 
with data pairs $\{(U_i^n, U_i^{n+1})\}_{i=1}^L$, which are NLMC solution 
coefficients (see Section \ref{sec:NLMC} for more details details 
of data generation) to \eqref{eq:flow} at $t^n$ and $t^{n+1}$, respectively. 
More precisely,  we consider the linear case of \eqref{eq:flow}, 
when $\kappa (t,x)$ is independent of $u_h^n$. In this case, the 
one step fluid dynamics \eqref{eq:recurrent_relation} is indeed a
 linear revolution such that $A^n$ is only dependent on time. 
Further, we denote $\hat{W_n}$ and $\hat{b_n}$ by
\begin{equation}
\begin{split}
\hat{W}& = (M + \Delta t A^n)^{-1} M, \\
\hat{b} & = (M + \Delta t A^n)^{-1} \Delta t F^n. \\
\end{split}
\label{eq:param}
\end{equation}
Then 
\begin{equation}
\label{eq:discrete}
U^{n+1}_i = \hat{W} U^n_i + \hat{b},\quad 1\leq i\leq L,
\end{equation}
for all training samples that we use in this section.
We further define the linear map between $U^n_i$ and $U_i^{n+1}$ as  $\mathcal{\hat{L}}(\cdot): \mathbb{R}^m\to \mathbb{R}^m$
\begin{equation}\label{eq:hat_L}
\mathcal{\hat{L}}(\mathbf{x}) := {\hat{W}} \mathbf{x} + {\hat{b}}.
%\qquad \mathcal{L}(\mathbf{x}) := W^2 W^1 \mathbf{x}+ W^2 b.
\end{equation}

From now on, when there is no risk of confusion, we drop the optimal network parameter $\theta^*$
in the trained network $\mathcal{NN}(\cdot)$. fz
We expect $\mathcal{NN}(\cdot)$ to learn the map $\mathcal{\hat{L}}(\cdot)$ from the data while extracting a system $W^2$  
in which the data $U^{n+1}_i$ can be represented sparsely.

\subsection{Sparsity and $\ell_1$ minimization}

To understand the trained neural network $\mathcal{NN}(\cdot)$,
we first assume the following.
\begin{assumption} \label{assum1}
For a subspace $S\subset \mathbb{R}^m$, there exist some orthogonal matrix $\hat{W}^2\in\mathbb{R}^{m\times m}$ such that $(\hat{W}^2)^T\mathcal{\hat{L}}(\mathbf{x})$ can be approximated by a sparse vector in $\mathbb{R}^m$ for any $\mathbf{x}\in S$.  More precisely, there exist an approximation error function $\epsilon(\cdot): \mathbb{N}\to  \mathbb{R}^+$, such that for any $\mathbf{x}\in S$, there exist a corresponding $s$-sparse vector $\mathbf{y}_s\in\mathbb{R}^m$ satisfies
\begin{equation}
\|(\hat{W}^2)^T\mathcal{\hat{L}}(\mathbf{x})-\mathbf{y}_s\|_2\leq \epsilon(s).
\end{equation}
\end{assumption} 
We then take $\{U_{i}^{n}\}_{i=1}^{L}$ from $S$ and define $U^{n+1, \hat{W}^2}_i: = (\hat{W}^2)^T U^{n+1}_i$ for all training pairs$ (U_i^n, U_i^{n+1})$ of $\mathcal{NN}(\cdot)$. By Assumption \ref{assum1}, there exist $s$-sparse vector $U^{s}_{i}\in \mathbb{R}^m$ such that
\begin{equation}\label{eq:w2_coordinates}
||U^{n+1, \hat{W}^2}_i - U^{s}_{i}||_2\leq \epsilon(s)
\end{equation}
for $1\leq i\leq L$. \\

Additionally, for any orthogonal matrix $W^2\in \mathbb{R}^{m\times m}$, we denote define $U_i^{n+1, {W}^2} := (W^2)^T U_i^{n+1}$ and denote it as $[\alpha_{i,1}^{n+1, W^2}, \alpha_{i,2}^{n+1,W^2}, \cdots, \alpha_{i,m}^{n+1,W^2}]^{T}$.
Recall \eqref{eq:sum_solution}, and the corresponding
numerical solution $u_{h,i}^{n+1}$ at time step $n+1$ can be written as
\begin{equation*}
u_{h,i}^{n+1} =\Phi U^{n+1}_i,
\end{equation*}
letting $\Phi =  [\phi_1, \phi_2,\cdots, \phi_m]$ be the multiscale basis functions, and $U^{n+1}_i$ be the corresponding coefficient vector.
By the orthogonality of $W^2$, $u_{h,i}^{n+1}$ can be further written as
\begin{equation}
u_{h,i}^{n+1}  = \Phi (W^2)  U^{n+1, W^2}_i = \Phi^{W^2}  U_i^{n+1, W^2},
\end{equation}
with $\Phi^{W^2}$ defined as
\begin{equation}\label{eq:basis_w2}
\Phi^{W^2}:  =  \Phi W^2 .
\end{equation}

We further denote the columns of $\Phi^{W^2}$ as $\phi^{W^2}_j$'s. Then $\{\phi^{W^2}_j\}_{j=1}^m$ is actually a new set of multiscale basis functions such that $u_{h,i}^{n+1} $ can be written as
\begin{equation}\label{eq:claim3}
u_{h,i}^{n+1} = \sum_{j=1}^{m} \alpha_{i,j}^{n+1,W^2}\phi^{W^2}_j.
\end{equation}
If $W^2$ is taken to be $\hat{W}^2 $ as in Assumption \ref{assum1},  we will obtain 

\begin{equation*}
u_{h,i}^{n+1} = \sum_{j=1}^{m} \alpha_{i,j}^{n+1,\hat{W}^2}\phi^{\hat{W}^2}_j,
\end{equation*}
where the coefficient $U^{n+1,\hat{W}^2}_i =[\alpha_{i,1}^{n+1, \hat{W}^2}, \alpha_{i,2}^{n+1,\hat{W}^2}, \cdots, \alpha_{i,m}^{n+1,\hat{W}^2}]^{T}$ can be closely approximated by a sparse vector $U_i^{s}$.

We then claim that our proposed neural network $\mathcal{NN}(\cdot)$ is able to approximate such $\hat{W}^2$ and a sparse approximation of the output from data.  This is guaranteed by the following lemma from \cite{beck2009fast}:
\begin{lemma}
	\label{lemma:3.1}
	We define $\mathcal{\hat{N}}(\cdot), \mathbb{R}^m \to \mathbb{R}^m$ as
	$$\mathcal{\hat{N}}(\mathbf{x})  :=S_{\gamma} (W\mathbf{x}+b).$$
	The output of $\mathcal{\hat{N}}(\mathbf{x}) $ is the solution of the $\ell_1$-regularized problem
	\begin{equation}
	\label{eq:l1}
	\mathbf{y}^*=\argmin_{\mathbf{y}\in \mathbb{R}^m} \dfrac{1}{2} \| \mathbf{y} - ({W} \mathbf{x} +{b}) \|_2^2 + \gamma \| \mathbf{y} \|_1
	\end{equation}
	by proximal gradient update.
	\begin{proof}
		It is straightforward to see that the directional derivative of the residual with respect to $\mathbf{y}$ is given by $  \mathbf{y}- ({W} \mathbf{x} + {b})$.
		On the other hand, the soft thresholding operator is the proximal operator of the $\ell_1$ norm, i.e.
		\[ S_\gamma(x) = \text{prox}_{\gamma \| \cdot \|_1} (\mathbf{x}) = \argmin_{\mathbf{y}\in\mathbb{R}^m} \left( \dfrac{1}{2} \| \mathbf{y} - \mathbf{x}\|_2^2+ \gamma \| \mathbf{y} \|_1 \right) . \]
		With a zero initial guess, the 1-step proximal gradient update of \eqref{eq:l1} with a step size $\gamma$ is therefore
		\begin{equation*}
		\begin{split}
		\mathbf{y}^* = \mathcal{\hat{N}}(\mathbf{x}) = S_\gamma({W} \mathbf{x} + {b}).
		\end{split}
		\end{equation*}
		%Stacking up the $d$ time steps, we obtain the result.
	\end{proof}
\end{lemma}
Thus, for the one-step neural network $\mathcal{NN}(\cdot)$ defined in \eqref{eq:one_layer}, Lemma \ref{lemma:3.1} implies that
\begin{equation}\label{lemma3.1_specify}
\mathcal{NN}(\mathbf{x}) = \argmin_{\mathbf{y}\in \mathbb{R}^m} \dfrac{1}{2} \| \mathbf{y} - W^2W^1\mathbf{x} + W^2b \|_2^2 + \gamma \| (W^2)^T\mathbf{y} \|_1.
\end{equation}
That is to say, the output of the trained neural network $\mathcal{NN}(\cdot)$ is actually the solution to a $\ell_1$ optimization problem. We further define the linear operator$\mathcal{{L}}(\cdot) : \mathbb{R}^m\to  \mathbb{R}^m$ as 
 \begin{equation}\label{eq:def_L}
\mathcal{{L}}(\mathbf{x}) := W^2W^1\mathbf{x} + W^2b.
\end{equation}
Equation actually \eqref{lemma3.1_specify} implies that 
\begin{equation}\label{eq:part1}
\mathcal{L}(\cdot) \approx \mathcal{NN}(\cdot)
\end{equation}
as $\mathcal{NN}(\mathbf{x})$ minimizes $ \| \mathbf{y} - W^2W^1\mathbf{x} + W^2b \|_2^2$. Moreover, the output of $\mathcal{NN}(\cdot)$ is sparse in the coordinate system $W^2$ as it also minimized the $\| (W^2)^T\mathbf{y} \|_1$ term.  \\

$S_{\gamma}$ is widely used in $\ell_1$-type optimization for promoting sparsity and extracting important features as discussed above. It is therefore also brought into network to extract sparsity from the training data. For other network  defined with activation functions such as ReLU, we also remark there's an correlation between $S_{\gamma}$ and ReLU. Recall its definition in \eqref{eq:s_gamma} :
\[S_\gamma(x) = \text{sign}(x) (\vert x \vert - \gamma)_+ \begin{cases}
x - \gamma & \text{ if } x \geq \gamma, \\
0 & \text{ if } -\gamma < x < \gamma, \\
x + \gamma & \text{ if } x \leq -\gamma.
\end{cases}\]
and that of
$\text{ReLU}: \mathbb{R} \to \mathbb{R}$
\[ \text{ReLU}(x) = \max\{x,0\} = \begin{cases}
x & \text{ if } x \geq 0, \\
0 & \text{ if } x < 0.
\end{cases} \]
We have the following: \\
\noindent One can explicitly represent the soft thresholding operator $S_\gamma$ by the ReLU function as
\begin{equation}
S_\gamma(x) = \text{ReLU}(x-\gamma)-\text{ReLU}(-x-\gamma) \quad  \forall x \in \mathbb{R},
\end{equation}
or in an entry-wise sense, one can write
\begin{equation}
S_\gamma(\mathbf{x}) = J_m \text{ReLU} (J_m^T \mathbf{x} - \gamma\mathbf{1}_{2m})  \quad  \forall \mathbf{x} \in \mathbb{R}^m,
\end{equation}
where $J_m = [I_m, -I_m]$.
Activation functions  $S_{\gamma}(\cdot)$ can thus be easily  implemented with the help of ReLU. Further, it also means that our proposed neural network is only a special class of neural networks that are defined with ReLU.

\subsection{linear operator $\mathcal{\hat{L}}\approx \mathcal{NN}$}

For neural network $\mathcal{NN}(\cdot)$ as defined in \eqref{eq:one_layer}
$$\mathcal{NN}(\mathbf{x}) = W^2S_{\gamma}(W^1 \mathbf{x}+b), $$ 
we claim the following

\begin{lemma}\label{calim1}
We assume Assumption \ref{assum1} holds. There exist a set of parameters $(W^1,W^2,b)\in\mathbb{R}^{m\times m}\times\mathbb{R}^{m\times m}\times\mathbb{R}^m$ such that
\begin{equation}\label{eq: lemma2}
\|\mathcal{NN}(\mathbf{x})-\mathcal{\hat{L}}(\mathbf{x})\|_2 \leq 2 \epsilon(s)+s^{\frac{1}{2}}\gamma, \; \quad\forall x\in S.
\end{equation}
\end{lemma}
\begin{proof}
		By Assumption \ref{assum1}, there exist some orthogonal matrix $\hat{W}^2\in\mathbb{R}^{m\times m}$  such that for all $\mathbf{x}\in S$, we have
\begin{equation}
\|(\hat{W}^2)^T\mathcal{\hat{L}}(\mathbf{x})-\mathbf{y}_s\|_2 \leq\epsilon(s),
\end{equation}
where $\mathbf{y}_s$ is an $s$-sparse vector.
Next, we consider $W^1 = (\hat{W}^2)^T \hat{W}$, $W^2 = \hat{W}^2$ and $b=(\hat{W}^2)^T \hat{b}$. We recall the definition of $\mathcal{\hat{L}}(\cdot)$
\begin{equation*}
\mathcal{\hat{L}}(\mathbf{x}) := {\hat{W}} \mathbf{x} + {\hat{b}}.
\end{equation*}
The difference between $\mathcal{NN}(\mathbf{x})$ and $\hat{\mathcal{L}}(\mathbf{x})$ can then be estimated by
\begin{align*}
\|\mathcal{NN}(\mathbf{x})-\hat{\mathcal{L}}(\mathbf{x})\|_{2} & =\|\hat{W}^{2}S_{\gamma}\Big((\hat{W}^{2})^{T}\hat{W}\mathbf{x}+(\hat{W}^2)^T \hat{b}\Big)-\hat{\mathcal{L}}(\mathbf{x})\|_2\\
 & =\|\hat{W}^2\Big(S_{\gamma}\Big((\hat{W}^{2})^{T}\hat{W}\mathbf{x}+(\hat{W}^2)^T\hat{b}\Big)-(\hat{W}^{2})^{T}\mathcal{\hat{L}}(\mathbf{x})\Big)\|_{2}\\
 & =\|\hat{W}^2\Big(S_{\gamma}\Big((\hat{W}^{2})^{T}\mathcal{\hat{L}}(\mathbf{x})\Big)-(\hat{W}^{2})^{T}\mathcal{\hat{L}}(\mathbf{x})\Big)\|_{2}\\
 & =\|S_{\gamma}\Big((\hat{W}^2)^T\mathcal{\hat{L}}(\mathbf{x})\Big)-(\hat{W}^{2})^{T}\hat{\mathcal{L}}(\mathbf{x})\|_{2}.\\
%& = \|S_{\gamma}\Big((\hat{W}^{2})^{T}\hat{W}\mathbf{x} + (\hat{W}^2)^T\hat{b}\Big)-\Big((\hat{W}^{2})^{T}\hat{W}\mathbf{x} +(\hat{W}^{2})^{T}\hat{b}\Big) \|_{2}
\end{align*}

Since $|S_{\gamma}(z_{2})-S_{\gamma}(z_{1})|\leq|z_{2}-z_{1}|,\;\forall z_{1},z_{2}\in\mathbb{R}$,
we have
\[
\|S_{\gamma}\Big((\hat{W}^2)^T\mathcal{\hat{L}}(\mathbf{x})\Big) - S_{\gamma}(\mathbf{y}_s)\|_{2}\leq \|(\hat{W}^2)^T\mathcal{\hat{L}}(\mathbf{x})-\mathbf{y}_s\|_2\leq \epsilon(s).
\]
%\[
%\|S_{\gamma}\Big((W^{2})^{T}\hat{W}\mathbf{x}\Big)-S_{\gamma}(\mathbf{y}_{s})\|_{2}\leq\|(W^{2})^{T}\hat{W}\mathbf{x}-\mathbf{y}_{s}\|_{2}\leq \epsilon(s).
%\]
Thus, we obtain
\begin{align*}
&\|S_{\gamma}\Big((\hat{W}^2)^T\mathcal{\hat{L}}(\mathbf{x})\Big)-(\hat{W}^{2})^{T}\hat{\mathcal{L}}(\mathbf{x})\|_{2} \\
\leq &\|S_{\gamma}\Big((\hat{W}^2)^T\mathcal{\hat{L}}(\mathbf{x})\Big) - S_{\gamma}(\mathbf{y}_s)\|_2+\| S_{\gamma}(\mathbf{y}_s )-\mathbf{y}_s  \|_2 +  \|\mathbf{y}_s-(\hat{W}^{2})^{T}\hat{\mathcal{L}}(\mathbf{x})\|_{2}  \\
\leq &2\epsilon(s) _+ \| \mathbf{y}_s -S_{\gamma}(\mathbf{y}_s ) \|.
\end{align*}
Since $|(S_{\gamma}(\mathbf{y}_{s})-\mathbf{y}_{s})_{i}|\leq\gamma$,
we have
\[
\|S_{\gamma}(\mathbf{y}_{s})-\mathbf{y}_{s}\|_{2}^{2}=\sum_{(\mathbf{y}_{s})_{i}\neq0}|(S_{\gamma}(\mathbf{y}_{s})-\mathbf{y}_{s})_{i}|^{2}\leq s\gamma^{2},
\]
and therefore we have
\[
\|\mathcal{NN}(\mathbf{x})-\hat{\mathcal{L}}(\mathbf{x})\|_{2}\leq 2\sigma(s)+s^{\frac{1}{2}}\gamma,
\]
letting $W^1 = (\hat{W}^2)^T \hat{W}$, $W^2 = \hat{W}^2$ and $b=(\hat{W}^2)^T \hat{b}$.
		%Stacking up the $d$ time steps, we obtain the result.
\end{proof}
%Through training, we actually obtained an operator  $\mathcal{NN}(\cdot) : \mathbb{R}^m\to \mathbb{R}^m$, with $\mathcal{NN} (\mathbf{x}) = W^2 S_\gamma (W^1 \mathbf{x} +b)$. 

Since $\mathcal{NN}(\cdot)$ is trained with $(U^{n}_i, U^{n+1}_i)$, where $U^{n}_i\in S $ and $\mathcal{\hat{L}}(U^{n}_i  ) =U^{n+1}_i $, we have 
\begin{equation}
\mathcal{NN}(\mathbf{x}) \approx \mathcal{\hat{L}}(\mathbf{x}), \quad \mathbf{x}\in S.
\end{equation}
 More specifically, this approximation error $\|\mathcal{NN}(\mathbf{x})-\mathcal{\hat{L}}(\mathbf{x})\|_2 $ is small for all $\mathbf{x}\in S$ providing sufficient training. Therefore, by Lemma \ref{calim1},  we claim the trained parameters closely approximate the optimal choice to guarantee the small error indicated in \eqref{eq: lemma2} , i.e.
\[W^2 \approx\hat{W}^2,\quad W^2W^1 \approx \hat{W}, \quad \text{and} \quad  W^2 b = \hat{b}.\]

%\begin{equation}
%U^{n+1}_i \approx \mathcal{NN}(U^n_i).
%\end{equation}
%Therefore, we say $U^{n+1}_i$ can be approximated with the solution to the optimization problem in \eqref{lemma3.1_specify}. The $\ell_1$ regularizer term indicates the solution is thus sparse in the $W^2$ system. In other words, if $\mathcal{NN}(\cdot)$ is trained with data under Assumption \ref{assum1}, it will be able to learn a matrix $W^2\approx \hat{W}^2$  and provide a sparse approximation of $U^{n+1}_i$ providing sufficient training.

However, due to the high dimension of $\hat{W}$, full recovery of $\hat{W}$ requires enormous number of training and is thus impractical.  However,  by enforcing $\mathcal{NN}(\mathbf{x}) =  \mathcal{\hat{L}}(\mathbf{x})$ for $\mathbf{x} \in S$, the neural network learns a set of parameters $W^2W^1 \neq \hat{W}$, and $W^2b \neq \hat{b} $ such that they functions similarly as $\mathcal{\hat{L}}(\cdot)$ on the subset $S$ in the sense of linear operator. A validation of this is later provided in Subsection \ref{section_same_function}. Recall definition of $\mathcal{L}(\cdot)$ in \eqref{eq:def_L} and the fact that it can approximate $\mathcal{NN}(\cdot)$ as in \eqref{eq:part1}, we claim the linear operator have the following property:
\begin{equation}\label{eq:part2}
\mathcal{{L}}(\mathbf{x}) \approx \mathcal{\hat{L}}(\mathbf{x}) \quad \forall\mathbf{x} \in S.
\end{equation}
In the following subsection, we further construct a reduced-order model with the help of  $\mathcal{L}(\cdot)$.

%\subsubsection{linear 1 }
%We further claim that the trained neural network provide us with a linear operator which can closely approximate the linear operation in \eqref{eq:discrete}. 

%\begin{equation}\label{linear_operator}
%\mathcal{L}(\mathbf{x}) := {W_n} \mathbf{x} + {b_n},\qquad  \mathcal{\hat{L}}(\mathbf{x}) := {\hat{W}_n} \mathbf{x} + {\hat{b}_n},
%\end{equation}
%and claim the following
%\begin{claim}\label{claim1}
%	\begin{equation}\label{eq:apprx}
%	\mathcal{\hat{L}}(\cdot) \approx \mathcal{N}^n(\cdot) \approx \mathcal{L}(\cdot)
%	\end{equation}
%	on the subspace $S \subset \mathbb{R}^m$.
%	%This implies
%	%\begin{equation}\label{eq:}
%	%\mathcal{L} (U^n) \approx \mathcal{\hat{L}} (U^n) = U^{n+1}.
%	%\end{equation}
%\end{claim}

\subsection{Model reduction with $W^{2}$}\label{model_reduction_section}

In this subsection, we further assume the $s$-sparse vector $\mathbf{y}_s$ in Assumption \ref{assum1} 
has non-zero entries only at fixed coordinates for all $ \mathbf{x}$ in $S$. That is to say, we have a fix reordering $\{j_k\}_{k= 1}^s$ for $\{1,2\cdots m\}$, such that $(\mathbf{y}_s)_{j_k}= 0$ for $ s+1 \leq k\leq m$.

Then, we will be able to utilize the coordinate system $W^2\approx \hat{W}^2$ learned through training network to construct a reduced-order operator $ \mathcal{L}_{s}(\cdot)$, such that it can approximate the linear map  $\mathcal{\hat{L}}(\cdot)$ and maps $\mathbf{x}$ in $ S$ to a $s$-sparse vector in $\mathbb{R}^m$. To do so, we first define $\mathcal{L}(\cdot)$ form learned coefficients of $\mathcal{NN}(\cdot)$ as in last subsection, and $\mathcal{L}_{s}(\cdot)$ will be exactly a truncation of it.

Moreover, let the new basis set $\{\phi_{j}^{W^2}\}_{j=1}^m$ be defined with trained coefficient $W^2$ as in \eqref{eq:basis_w2}. When truncating $W^2$ in  $\mathcal{L}(\cdot)$, we also determine the dominant basis among  $\{\phi^{W^2}_j\}_{j=1}^m$ simultaneously. Thus, we can view the model reduction from another aspect that we actually drop the basis with less significance and represent the solution with only the dominant multiscale modes.

To  construct such $\mathcal{L}_{s}(\cdot)$, we follow the steps:
\begin{enumerate}
	\item Find the dominant coordinates of outputs $\mathcal{{NN}}(\cdot )$ in the system $W^2$.
	\begin{enumerate}
		\item Compute $W^2$ system coefficients of $\mathcal{{NN}}(U^{n}_i)$ for all training samples by 
		\begin{equation*}
		O^{n+1, W^2}_i:= (W^2)^T \mathcal{{NN}}(U^{n}_i), \quad 1\leq i\leq L,
		\end{equation*}
		where $i$ refers to the sample index. Notice $(W^2)^T\mathcal{NN}(\mathbf{x})$ is sparse for $\mathbf{x}\in S$ by \eqref{lemma3.1_specify}, therefore $O^{n+1, W^2}_i$ is also sparse. 
		\item Calculate the quadratic mean of $\{O^{n+1, W^2}_i\}_{i=1}^{L}$ over all samples, coordinate by coordinate:
		$$S_j := \frac{1}{L}(\sum_{i= 1}^L |O^{n+1, W^2}_{i,j}|^2)^{\frac{1}{2}},\quad 1\leq j \leq m.$$
		\item Sort the quadratic mean value $S_j$ in descending order and denoted the reordered sequence as $\{S_{j_k}\}_{k=1}^{m}$.
	
	\end{enumerate}
	\item Keep the dominant $j_k$-th columns of $W^2$ for $k= 1, \cdots, s$. Then let the rest columns be zero. Thus, we construct a reduced-order coordinate system $W^{2,s}\in \mathbb{R}^{m\times m}$.
	Consequently,  $\mathbf{y} = \mathcal{\hat{L}}(\mathbf{x})$ for any $\mathbf{x}\in S$ can be approximated with the reduced-order system $W^2$ as an $s$-sparse vector $\mathbf{y}^{ W^{2,s}} : = (W^{2,s})^T\mathbf{y}$, and
	\begin{equation}
	\mathbf{y}^{ W^{2,s}} \approx  (W^{2})^T \mathbf{y}.
	\end{equation}
	 Consequently,  for training/testing samples, we have ${U}^{n+1, W^{2,s}}_i \approx 	{U}^{n+1, W^{2}}_i$.  Thus,  corresponding function  $u_{h,i}^{n+1} $ can be approximated with only basis $\{\phi_{j_k}^{W_2}| 1\leq k\leq s\}$ correspond to dominant multiscale modes.
	\begin{equation}
	u_{h,i}^{n+1}  \approx \Phi^{W^2}  U_i^{n+1, W^{2,s}} = \sum_{k=1}^{s} \alpha_{i,j_k}^{n+1,W^2}\phi^{W^{2}}_{j_k}.
	\end{equation}
	
	\item We finally define the reduced linear operator $\mathcal{L}_s(\cdot): \mathbb{R}^m\to \mathbb{R}^m$ as
	\begin{equation}\label{reduce_claim}
	\mathcal{L}_s(\mathbf{x}) = W^{2,s}W^1 \cdot \mathbf{x} + W^{2,s}b, \quad \mathbf{x}\in \mathbb{R}^m.
	\end{equation}
	Here, the output of $\mathcal{L}_s(\cdot)$ is an $s$-sparse vector in $\mathbb{R}^m$.
\end{enumerate}
		
This algorithm is designed based on the fact that  $O_i^{n+1,W^2} \approx (W^{2})^TU^{n+1}_i = U^{n+1,W^2}_i $ as $\mathcal{NN}(\cdot )$ is fully trained with $(U^{n}_i, U^{n+1}_i)$.  Thus, the order of $S_j$ can reflect not only the significance of the coordinates of the output of $\mathcal{NN}(\cdot)$  but also that of $(W^2)^T \mathcal{L}(\mathbf{x})$ for all $\mathbf{x}$ in $S$.  Moreover, the existence of the sparse approximation $\mathbf{y}_s$ to $(\hat{W}^2)^T\mathcal{\hat{L}}(\mathbf{x})$ as described in Assumption \ref{assum1} guarantees the effectiveness of the ordering.

We then claim that this reduced-order linear operator $\mathcal{L}_s(\cdot)$ can approximate the true input-output map $\mathcal{\hat{L}}(\cdot)$ on $S$:
%\begin{claim}\label{claim4}
	Since 	$\mathcal{L}_s(\cdot)$ is simply a truncation of $\mathcal{L}(\cdot)$,  we have:
	\begin{equation}
	\mathcal{L}_s \longrightarrow\mathcal{L},\quad \text{as}\quad s \rightarrow m.
	\end{equation}
	Moreover, recall \eqref{eq:part2}
	\begin{equation*}
\mathcal{\hat{L}}(\mathbf{x}) \approx \mathcal{{L}}(\mathbf{x}),\quad \forall \mathbf{x} \in S,
	\end{equation*}
	it implies the following
	\begin{equation}\label{reduce_verify}
	\mathcal{\hat{L}}(\mathbf{x}) \approx\mathcal{L}_s(\mathbf{x}), \quad \forall \mathbf{x} \in S.
	\end{equation}
%\end{claim}

This property of $\mathcal{L}_s(\cdot)$ provides us a way to represent the projected vectors $\mathcal{\hat{L}}(\mathbf{x})$ for $\mathbf{x} \in S$ using a vector with only $s$ nonzero coefficients, which corresponds to a reduced multiscale model to represent the class of solution $u_h^{n+1}$ that we are interested in.

Numerical examples are presented in Section \ref{section_W_2_R} to verify this claim, from which we actually observe that $s$ can be taken as a fraction of the original number of multiscale basis $m$ to give the approximation in \eqref{reduce_verify}.

\section{Numerical experiment}\label{section_example}
In this section, we present numerical examples in support of the previous discussion on the reduced-order neural network. Specifically, Section \ref{section_same_function} demonstrates the $\mathcal{L}(\cdot)$ and $\mathcal{\hat{L}(\cdot)}$ functions similarly on a subspace by comparing the eigenvalues of the two operators; Section \ref{section_one_layer} demonstrates that a one-layer soft thresholding neural network can accurately recover a linear dynamics with a sparse coefficient vector; Section \ref{section_W_2_R} then uses the learned coefficient $W^2$  from the one-layer neural network to conduct model reduction as describe in Section \ref{model_reduction_section}; Section \ref{multi-layer} later presents the predicting results for multi-layer reduced-order neural network which corresponds to Section \ref{sec:full_nn}; and Section \ref{section_clustering} applies the clustering scheme to nonlinear process modeling.
All the network training are performed
using the Python deep learning API TensorFlow \cite{tensorflow2015-whitepaper}.
%\subsection{Reduced-order neural network}
\subsection{$\mathcal{\hat{L}} \approx\mathcal{L}$ in a subspace of $\mathbb{R}^m$}\label{section_same_function}

We recall the one-layer neural network for a single-step linear dynamics in \eqref{eq:one_layer}
\begin{equation*}
\mathcal{NN}(\mathbf{x}): = W^2S_\gamma(W^1 \mathbf{x} + b),
\end{equation*}
and the definition of $\mathcal{\hat{L}}(\cdot)$ and $\mathcal{L}(\cdot)$ in \eqref{eq:hat_L} and \eqref{eq:def_L} respectively:
\begin{equation*}
 \mathcal{\hat{L}}(\mathbf{x}) := {\hat{W}} \mathbf{x} + {\hat{b}},\qquad
\mathcal{L}(\mathbf{x}) = W^2 W^1 \mathbf{x}+ W^2 b,
\end{equation*}
where ${\hat{W}_n}$ and ${\hat{b}_n}$ are defined as in \eqref{eq:param}, while $W_n^2, W_n^1$ are trained parameters of $\mathcal{NN}(\cdot)$.  We also recall \eqref{eq:part2}:
\begin{equation*}
\mathcal{{L}}|_{S} \approx \mathcal{\hat{L}}|_{S},
\end{equation*}
for $S\subset \mathbb{R}^m$.

To support this claim, we design a special subspace $S\subset \mathbb{R}^m$.
For $r<m$, we then let
\begin{equation}
S = V_r := \text{span}\{v_i, 1\leq i\leq r\} \quad\subset\quad \text{span}\{v_i, 1\leq i\leq m\} = \mathbb{R}^m,
\end{equation}
where $\{v_i\}_{i=1}^m$ are eigenvectors of $\hat{W}$ corresponding to eigenvalues $\lambda_i$ in descending order. We also define matrix $V$ as
\begin{equation}
V := [v_1, v_2,\cdots, v_m].
\end{equation}

We then randomly pick training input $U \in V_s$ such that $U = \sum_{i=1}^{r} c_iv_i$. The $\mathcal{NN}(\cdot)$ is then trained with $(U, \mathcal{\hat{L}}(U))$-like training pairs, and we obtain a corresponding operator $\mathcal{L}(\cdot) $ with trained coefficients. The linear operators of $\mathcal{L}(\cdot) $ and $\mathcal{\hat{L}}(\cdot) $ are compared by their eigenvalues,  i.e. $V^T \hat{W} V $ and $V^T W^2W^1 V $. By the definition of $V$, the former will exactly be a diagonal matrix with $\lambda_i$ be its diagonal value. We expect $W^2W^1$ functions similarly to $\hat{W} $ on $V_r$, and further the $r$-by-$r$ sub-matrix of $V^T W^2W^1 V $ should be similar to that of $V^T \hat{W} V$.

Figure \ref{fig:hat_L_L} compares $V^T W^2W^1 V $ and $V^T \hat{W} V$ for the case when $V_r$ is constructed letting $r= 30$. We can tell that the first $30\times 30 $ submatrix are very much alike. That is to say, despite the fact that the operator $\mathcal{L}(\cdot)$ and $\mathcal{\hat{L}}(\cdot)$ are different on $\mathbb{R}^m$, their behavior on the subspace $V_r$ are the same. Moreover, Figure \ref{fig:eigen-value}, shows that such similarity only exist in $V_r$ as for the $i_{th}$ diagonal values of $V^T W^2W^1 V $ and that of $V^T \hat{W} V$ distinct when $i>s$. This also makes sense as the operator $\mathcal{L}(\cdot)$ is defined from the trained parameters of $\mathcal{NN}(\cdot)$ where only subspace $V_r$ is visible to the network.

\begin{figure}[htbp]
	%	\subfigure[U1 coordinate]{./fig/U1_coordinate.png}
	\centering
	
	\begin{subfigure}[t]{0.35\textwidth}
		\includegraphics[width=\textwidth]{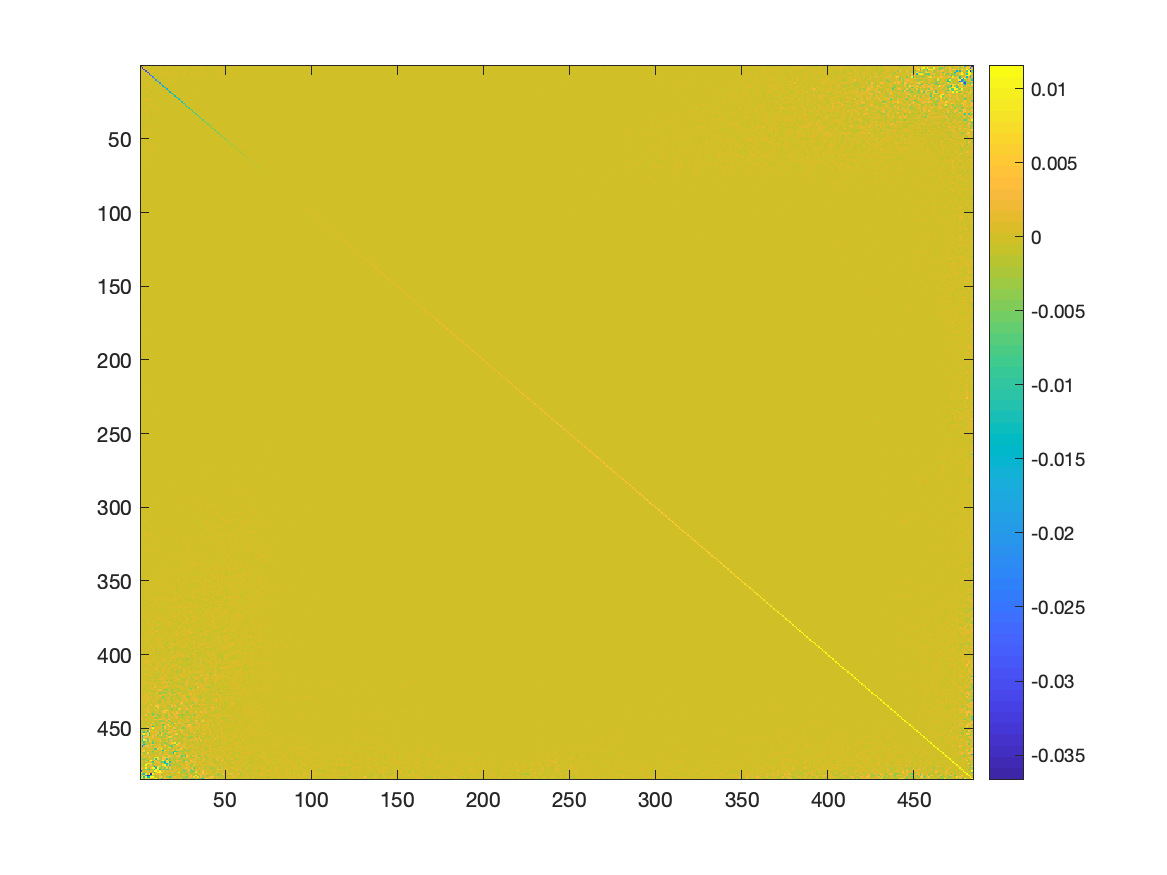}
		\caption{$V^T\hat{W}V $}
		\label{fig:ture_whole}
	\end{subfigure}
	~
	\begin{subfigure}[t]{0.35\textwidth}
		\includegraphics[width=\textwidth]{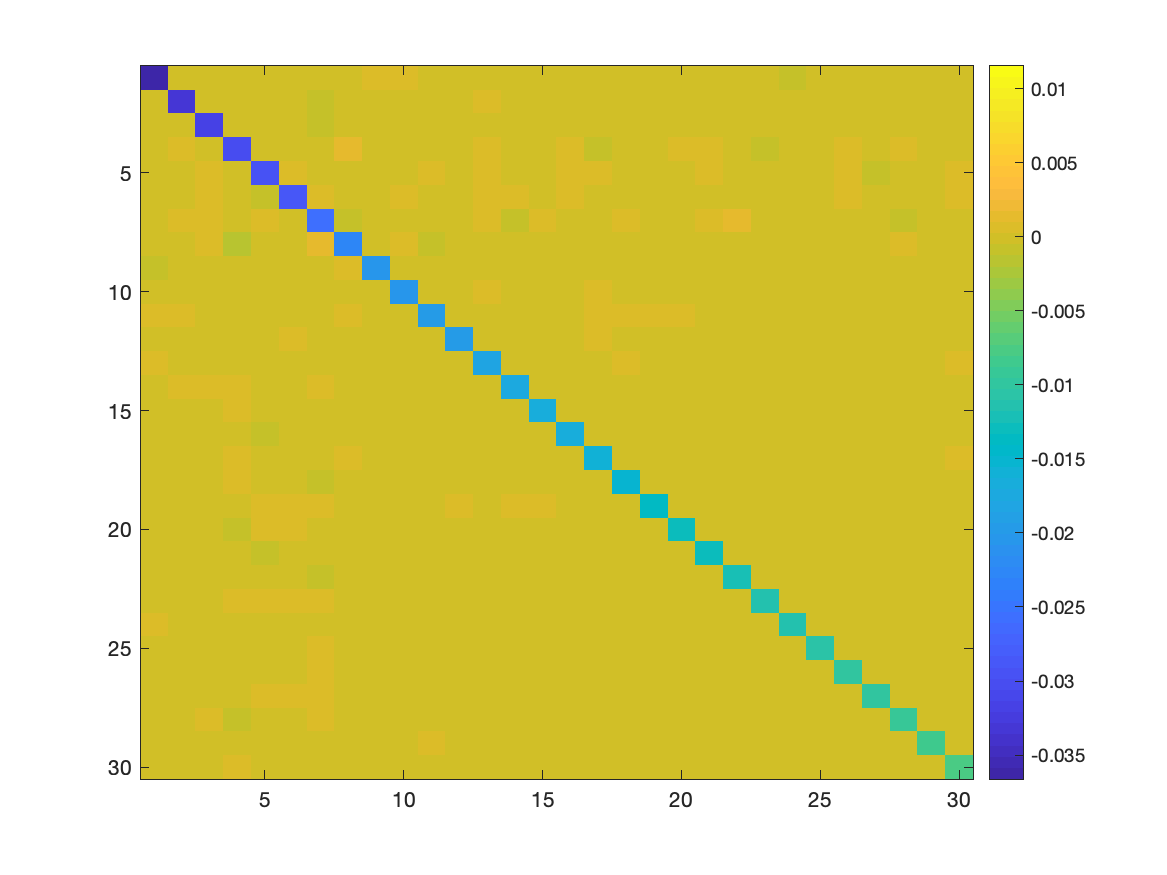}
		\caption{$30\times30$ sub-matrix of $V^T\hat{W}V $ }
		\label{fig:30_true}
	\end{subfigure}

	\begin{subfigure}[t]{0.35\textwidth}
		\includegraphics[width=\textwidth]{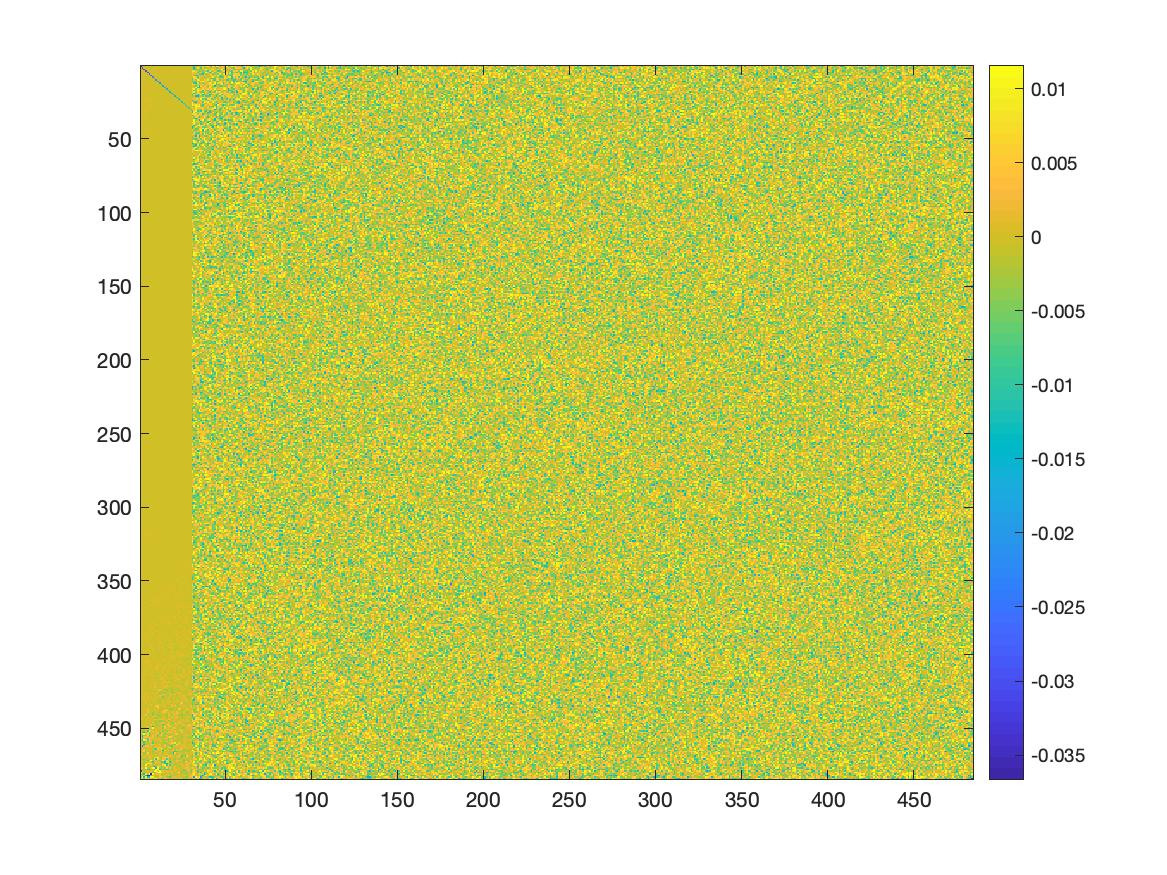}
		\caption{$V^T W^2W^1 V $}
		\label{fig:learn_whole}
	\end{subfigure}
	~
	\begin{subfigure}[t]{0.35\textwidth}
		\includegraphics[width=\textwidth]{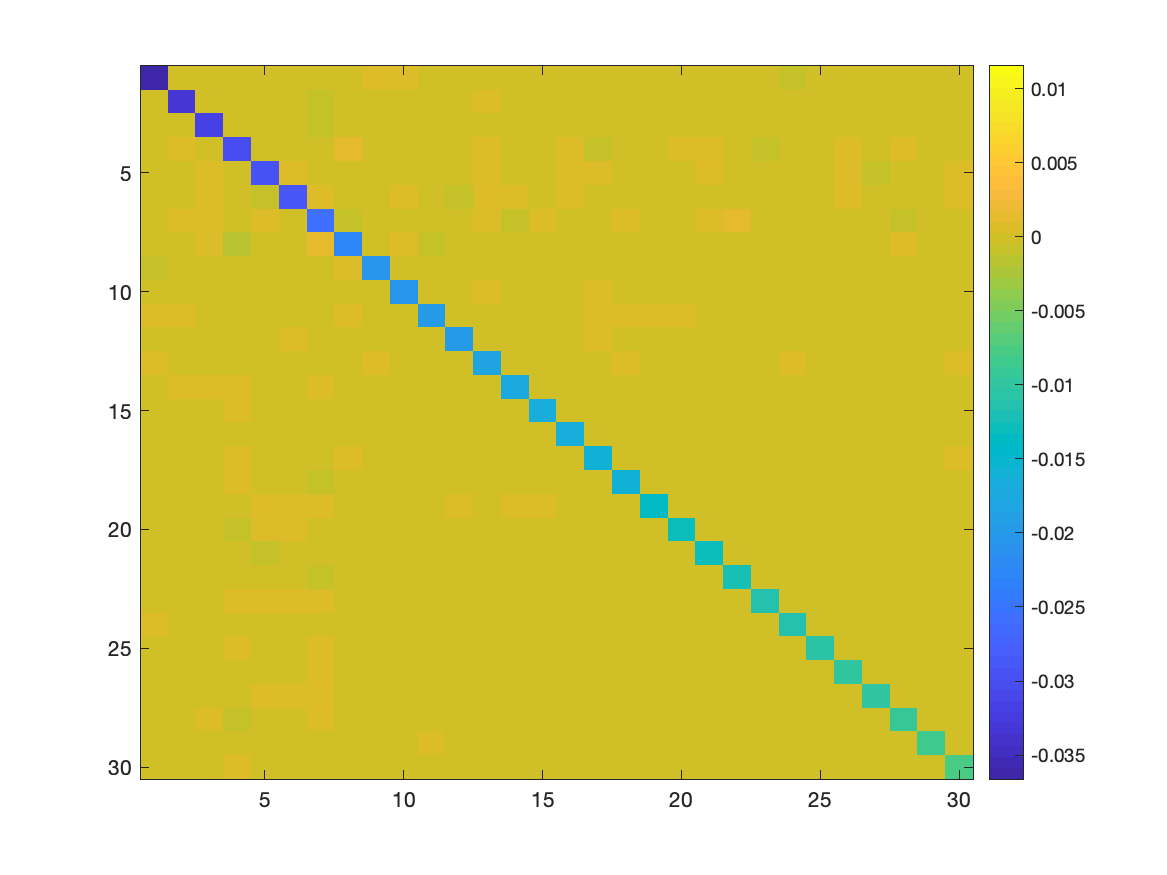}
		\caption{$30\times30$ sub-matrix of $V^T W^2W^1 V $ }
		\label{fig:30_learn}
	\end{subfigure}
	~	\caption{$\hat{W}$ and $W^2W^1$ function similarly on $S$, where $r = 30$.}
	\label{fig:hat_L_L}
\end{figure}

\begin{figure}[htbp]
	\centering
	\includegraphics[width= 0.5\textwidth]{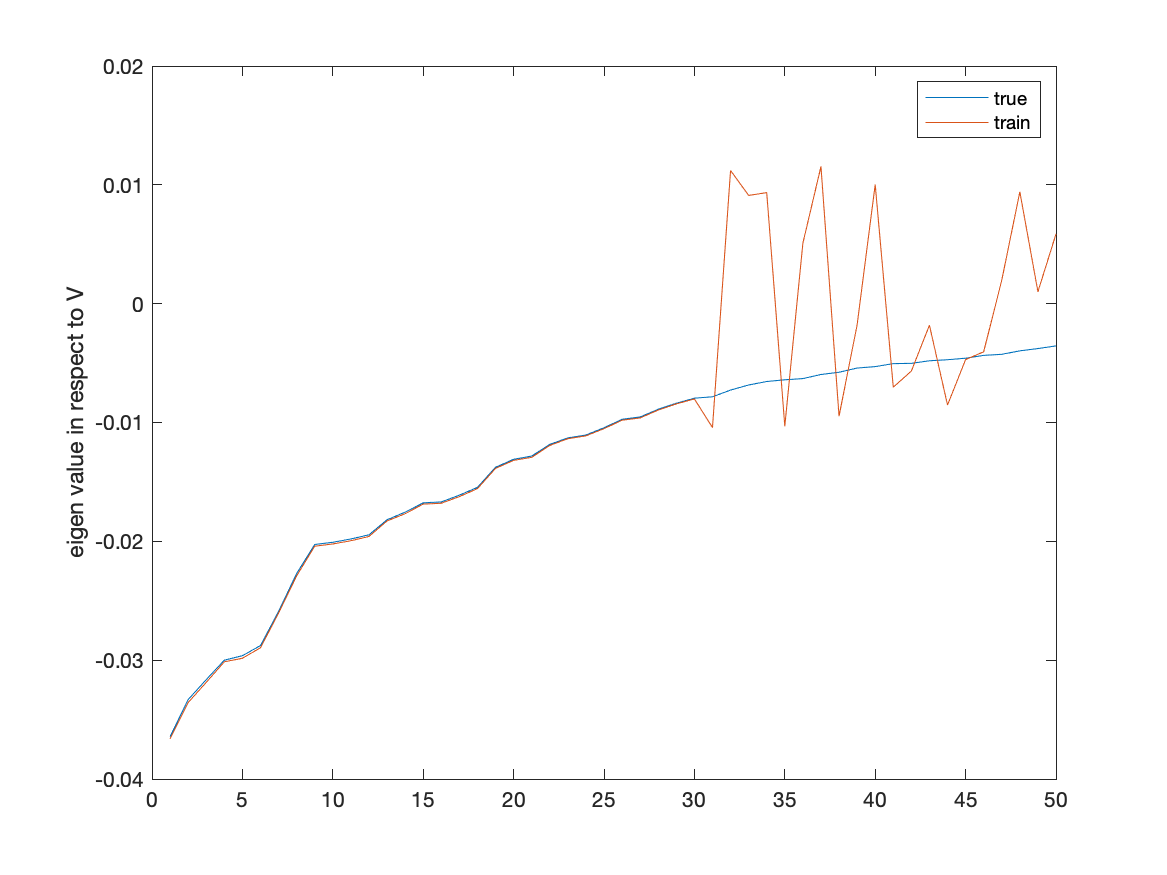}
	\caption{Comparison of eigen-values of $\hat{W}$ and $W^2W^1$  when $r = 30$.}
	\label{fig:eigen-value}
\end{figure}

\newpage
\subsection{One-layer reduced order neural network}\label{section_one_layer}
In this example, we consider the one-layer reduced-order neural network as defined in \eqref{eq:one_layer}. We use this neural network to predict a one-step fluid dynamics, where the data are taken to be NLMC solution coefficients to \eqref{eq:flow} in the form $(U^0_i,U^1_i )$. We fix $\kappa(t,x)$ and $f(t, x)$ among samples, thus all data describes linear dynamics for different initial conditions.

We take $2\%$ out of all data pairs as testing samples and the remaining $98\%$ as training samples and use only the training sample to train $\mathcal{NN}(\cdot)$. We then evaluate the neural network by examining the accuracy of the following approximation for the unseen testing samples:
$$U^1\approx \mathcal{NN}(U^0).$$
The $\ell_2 $ relative error of the prediction is computed by
\begin{equation}\label{l2_error}
\frac{||U^1- \mathcal{NN}(U^0)||_{2}}{||U^1||_{2}}.
\end{equation}

Table \ref{tab:single_layer_prediction_error} is the error table for the case when we use $500$ data pairs with $490$ to be training samples and $10$ to be testing samples. These data are generated with different choice of initial condition $U_i^0$. To match the realistic physical situation, we took all initial conditions to be the NLMC terminal pressure of a mobility driven nonlinear flow process.

\begin{table}[htbp]
 \centering

    \begin{tabular}{|c|r|}
    	\hline
    Sample Index& Error(\%)\\
    	\hline
    \#1    & 0.25 \\
     \#2    & 0.43 \\
      \#3     & 10.02 \\
      \#4     & 9.91 \\
     \#5    & 3.90 \\
      \#6     & 8.18 \\
    \#7   & 17.27 \\
    \#8  & 1.57 \\
     \#9   & 1.13 \\
   \#10  & 0.76 \\
    \hline
      Mean    & 5.34 \\
          \hline
    \end{tabular}%
  \caption {$\ell_2$ relative error of $\mathcal{NN}(\cdot)$ prediction.}
  \label{tab:single_layer_prediction_error}%
\end{table}%

From Table \ref{tab:single_layer_prediction_error}, we can see that the prediction of our proposed network $\mathcal{NN}(U^0)$ is rather effective with an average $\ell_2$ error of $5.34\%$.

We also verify that $U^1$ is sparse in the learned $W^2$-system for all data(training and testing). We first reorder the columns of $W^2$ by their dominance as discussed in Section \ref{model_reduction_section}, then compute the corresponding $W^2$-system coefficients $U^{1,W^2}$, which should be a roughly decreasing vector. From Figure \ref{one_step_sparse}, we can tell that the $W^2$-system coefficients $U^{1,W^2}$ are sparse. This can be an reflection of successful learning of $\hat{W^2}$ in Assumption \ref{assum1}. Moreover, only a few dominant modes are needed to recover the solution as the quadratic averages of coordinates $|U^{1,W^2}_j|$ decays fast when $j >100$.

\begin{figure}[htbp]
	%	\subfigure[U1 coordinate]{./fig/U1_coordinate.png}
	\centering
	\begin{subfigure}[b]{0.45\textwidth}
		\includegraphics[width=\textwidth]{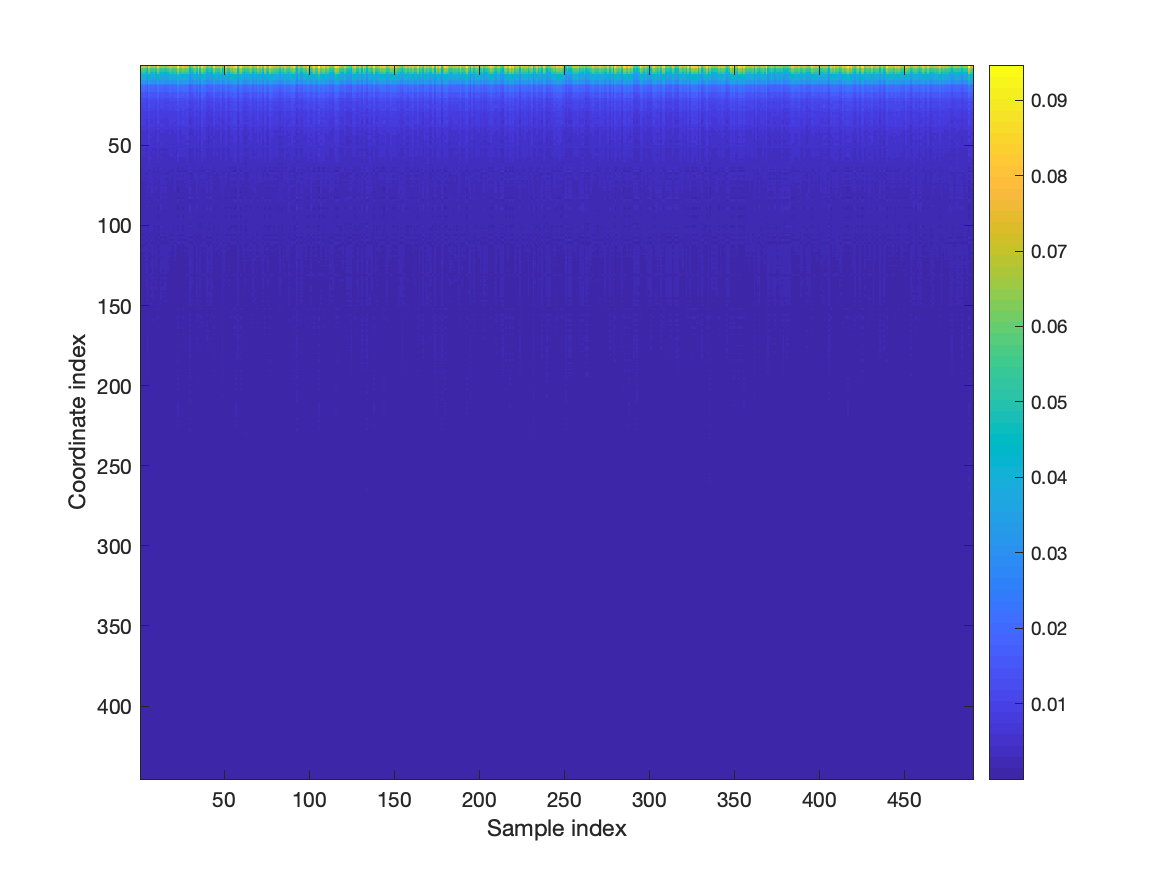}
		\caption{$U^{1,W^2}$ for all data samples }
		\label{fig:gull}
	\end{subfigure}
	~
	\begin{subfigure}[b]{0.45\textwidth}
		\includegraphics[width=\textwidth]{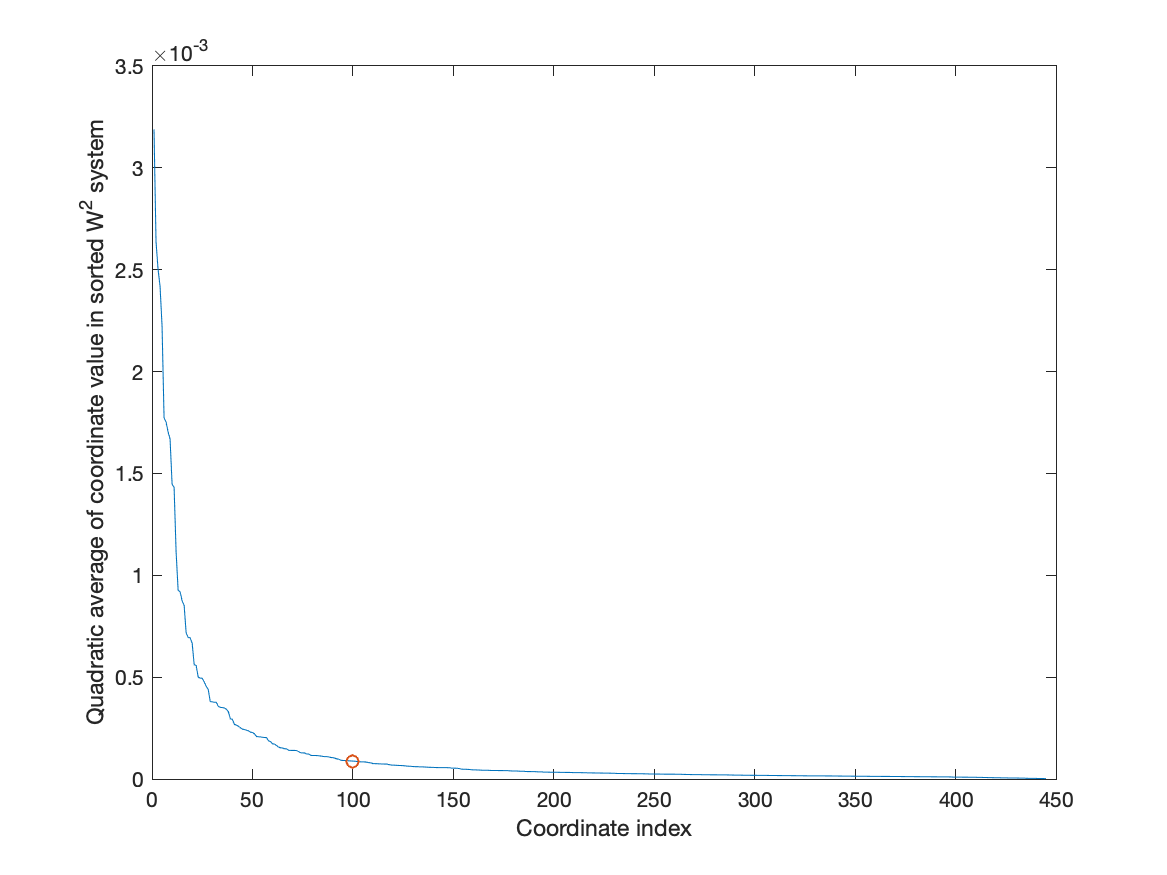}
		\caption{ quadratic average of $U^{1,W^2}$ among all data samples}
		\label{fig:gull}
	\end{subfigure}
	~
	\caption{Sparsity of the output of $\mathcal{NN}(U^0)$ in $W^2$ .}
	\label{one_step_sparse}
\end{figure}

In a word, the proposed neural network can indeed learn the dominant multiscale modes needed to represent $U^{1}$ from training data while properly reproduce the map between $U^0$ and $U^1$.

\subsection{ Model reduction with $W^2$}\label{section_W_2_R}
As discussed in Section \ref{model_reduction_section}, we would like to used the reduced-order system $W^{2,s}$ to further conduct model reduction. The reduced-order solution coefficient is defined with the reduced-order linear operator $\mathcal{L}_s(\cdot)$:
\begin{equation}
U^{n+1}_{sN} =\mathcal{L}_s (U^n).
\end{equation}
Noticing that $U^{n+1}_{sN}$ is the coefficient of $u_h^{n+1}$ in the original basis system $\{\phi_j\}_{j=1}^{m}$ while it is sparse in $W^2$-system, i.e, $U^{n+1,W^2}_{sN}$ is sparse and has maximum $s$ nonzero elements.

The numerical experiments is conduced based on a one-layer neural network as defined in \eqref{eq:one_layer} for a one-step linear dynamics. We would like to compare the following coefficient vectors:
$$U^1_{sN} :=\mathcal{L}_s (U^0) = W^{2,s} W^1\cdot U^0+ W^{2,s}b, $$
$$U^1_{\text{true}} = \hat{W}U^0+  \hat{b},$$
and
 $$U^1_{N} := \mathcal{NN}(U^0).$$
 Here $U^1_{\text{true}} $ is the true solution to \eqref{eq:discrete}, while
$U^1_{N}$ is the prediction of $\mathcal{NN}(\cdot)$.

\begin{figure}[htbp]
	\centering
	\includegraphics[width=0.5\textwidth]{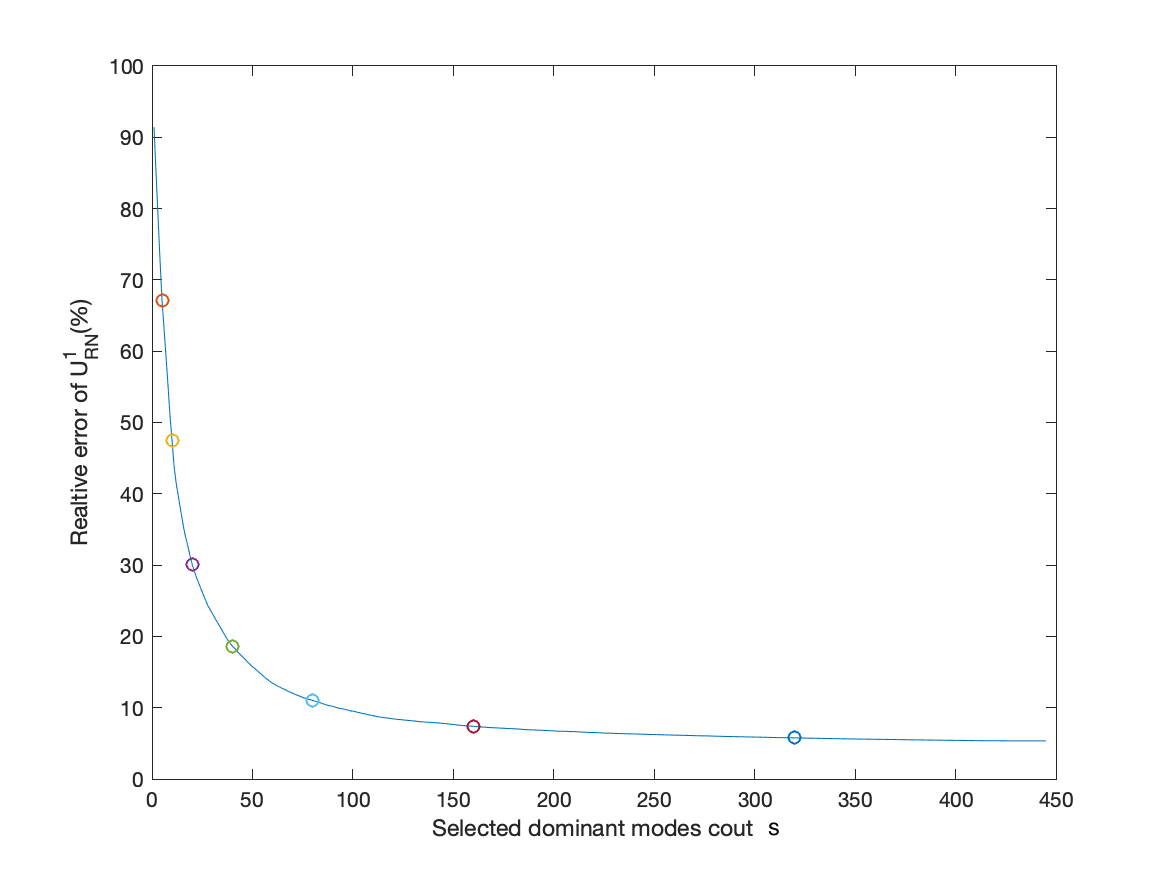}
	\caption{Decay of relative error $\frac{||U^1_{sN}-U^1_{\text{true}}||_{2}}{||U^1_{\text{true}}||_{2}}$ as $s$ grows.}
	\label{fig:error_decay}
\end{figure}

\begin{table}[htbp]
	\centering
	\begin{tabular}{|c|r|r|r|r|r|r|r|}
		\hline
		\diagbox{Sample Index}{$s$}  & 5     & 10    & 20    & 40    & 80    & 160   & 320 \\
		\hline
		\#1    & 66.71 & 46.80 & 29.02 & 17.01 & 8.44  & 3.53  & 1.11 \\
		\#2   & 66.55 & 46.58 & 29.03 & 17.41 & 9.48  & 4.62  & 1.59 \\
		\#3   & 66.35 & 46.77 & 29.70 & 18.59 & 12.26 & 10.92 & 10.16 \\
		\#4   & 67.95 & 48.58 & 31.23 & 19.87 & 12.59 & 10.03 & 9.90 \\
		\#5   & 66.22 & 46.31 & 28.99 & 17.73 & 10.49 & 6.49  & 4.34 \\
		\#6   & 66.47 & 46.76 & 29.40 & 17.89 & 10.86 & 9.12  & 8.33 \\
		\#7   & 69.59 & 51.33 & 36.01 & 27.42 & 22.33 & 18.62 & 17.29 \\
		\#8   & 66.60 & 46.61 & 28.81 & 16.39 & 7.47  & 3.86  & 1.96 \\
		\#9    & 66.90 & 47.13 & 29.21 & 16.94 & 8.02  & 3.58  & 1.67 \\
		\#10    & 67.14 & 47.39 & 29.44 & 17.23 & 8.23  & 3.19  & 1.32 \\
		\hline
		Mean& 67.05 & 47.43 & 30.08 & 18.65 & 11.02 & 7.40  & 5.77 \\
		\hline
	\end{tabular}%
	\caption{Decay of error $\frac{||U^1_{sN}-U^1_{\text{true}}||_{2}}{||U^1_{\text{true}}||_{2}}$ with respect to number of selected dominant modes in $W^{2,s}$. }
	\label{tab:error_decay}%
	
\end{table}%

Figure \ref{fig:error_decay} shows the error decay of $U^{1}_{sN}$ compared to $U^{1}_{\text{true}}$. As $s$ grows, the error gets smaller. This figure actually verified \eqref{reduce_verify}, i.e. the reduced operator $\mathcal{L}_s(\cdot)$ can approximate $\mathcal{\hat{L}}(\cdot)$.Moreover, this approximation gets more accurate  as $s$ gets larger. The error in Figure \ref{fig:error_decay} at $s =m= 445$ is also expected that it is a consequence of the training error. Besides, we observe that the error decays fast when $s >40$ for our training samples.

Table \ref{tab:error_decay} further facilitates such conclusion. The error of $U^{1}_{sN}$ is less than $12\%$ when $s \geq 80$. We can thus represent the multiscale solution $u_h^1$ using $s =80$ basis with little sacrifice in the solution accuracy. We notice that the order of the reduce operator $\mathcal{L}_s(\cdot)$  is only around $18\%$ that of the original multiscale model.

We lastly present the comparison between $U^{1}_{\text{true}}$, $U^{1}_{N}$, and $U^{1}_{sN}$.From Table \ref{tab:comparsion_error}, we can tell that for a single testing sample, we have
$$\frac{||U^1_{sN}-U^1_{\text{true}}||_{2}}{||U^1_{\text{true}}||_{2}}> \frac{||U^1_{N}-U^1_{\text{true}}||_{2}}{||U^1_{\text{true}}||_{2}}, \quad \frac{||U^1_{sN}-U^1_{\text{true}}||_{2}}{||U^1_{\text{true}}||_{2}}> \frac{||U^1_{sN}-U^1_{N}||_{2}}{||U^1_{N}||_{2}},$$
which is as expected since $\frac{||U^1_{N}-U^1_{\text{true}}||_{2}}{||U^1_{\text{true}}||_{2}} $and $\frac{||U^1_{sN}-U^1_{N}||_{2}}{||U^1_{N}||_{2}}$ are exactly the two components of error $\frac{||U^1_{sN}-U^1_{\text{true}}||_{2}}{||U^1_{\text{true}}||_{2}}$.  They stands for neural network prediction error and $\mathcal{L}_s$ truncation error, respectively. The latter can be reduce by increasing $s$, while the former one can only be improved with more effective training.

\begin{table}[htbp]
	\centering
	\begin{tabular}{|c|r|r|r|}
		\hline
		Sample Index &$\frac{||U^1_{N}-U^1_{\text{true}}||_{2}}{||U^1_{\text{true}}||_{2}}$ & $\frac{||U^1_{sN}-U^1_{\text{true}}||_{2}}{||U^1_{\text{true}}||_{2}}$& $\frac{||U^1_{sN}-U^1_{N}||_{2}}{||U^1_{N}||_{2}}$ \\
		\hline
	 \#1    & 0.25  & 6.48  & 6.41 \\
	 \#2   & 0.43  & 7.65  & 7.55 \\
	 \#3   & 10.02 & 11.59 & 5.81 \\
	 \#4   & 9.91  & 11.28 & 5.83 \\
	 \#5   & 3.90  & 8.93  & 8.38 \\
	 \#6   & 8.18  & 10.00 & 5.80 \\
	 \#7  & 17.27 & 20.98 & 10.25 \\
	 \#8   & 1.57  & 5.85  & 5.83 \\
	 \#9    & 1.13  & 6.13  & 6.03 \\
	 \#10    & 0.76  & 6.17  & 6.12 \\
	\hline
	Mean & 5.34  & 9.50 & 6.80 \\
		\hline
	\end{tabular}%
	\caption{ Relative error percentage of solutions obtained in full $W^2$ system and reduced-order system $W^{2,s}$ for $s = 100$. }
	\label{tab:comparsion_error}%
\end{table}%

\subsection{Multi-layer reduced order Neural Network}\label{multi-layer}
In this example, we use a multi-layer reduced-order neural network $\mathcal{NN}(\cdot)$ to predict multi-step fluid dynamics. Recall \eqref{eq:w2_multi_layer}, it is defined as
\begin{equation*}
\mathcal{NN}(\mathbf{x}^0) := \mathcal{N}^n(\cdots \mathcal{N}^1(\mathcal{N}^0(\mathbf{x}^0))).
\end{equation*}
The input of $\mathcal{NN}(\cdot)$ is taken to be $U^0$, the initial condition, while the outputs are the collection of outputs at $n_{th}$-layer sub-network $\mathcal{N}^n(\cdot)$ which correspond to the true values [$U^1, U^2, \cdots, U^9$].  $U^{n+1}$ are all taken to be NLMC solutions of \eqref{eq:flow} at time step $n$ for $n =0,\cdots, 8$. Prediction accuracy is measured with $\ell_2$ relative error that defined similar to \eqref{l2_error}.

% Table generated by Excel2LaTeX from sheet 'error_table'
\begin{table}[htbp]
  \centering
    \begin{tabular}{|c|r|r|r|r|r|r|r|r|r|}
    	          \hline
        Sample Index  &    $U^1$   &     $U^2$  &     $U^3$   &   $U^4$ &   $U^5$ &     $U^6$   &     $U^7$   &   $U^8$    &       $U^9$   \\
          \hline
    \#1    & 1.62  & 1.17  & 1.69  & 1.91  & 1.95  & 1.92  & 1.92  & 1.91  & 1.96 \\
     \#2   & 3.33  & 1.86  & 2.10  & 1.98  & 2.15  & 1.53  & 1.04  & 0.94  & 0.77 \\
     \#3   & 11.62 & 13.32 & 9.39  & 9.57  & 8.89  & 10.36 & 11.67 & 11.86 & 12.97 \\
     \#4   & 9.74  & 9.00  & 4.21  & 3.45  & 3.46  & 3.17  & 2.93  & 2.87  & 2.81 \\
     \#5   & 5.63  & 4.68  & 2.65  & 2.28  & 2.52  & 1.89  & 1.49  & 1.74  & 1.91 \\
     \#6   & 9.54  & 11.50 & 9.30  & 9.70  & 9.14  & 10.50 & 11.73 & 12.00 & 13.09 \\
     \#7   & 21.46 & 14.82 & 5.42  & 4.24  & 3.27  & 5.06  & 6.52  & 6.77  & 7.81 \\
     \#8   & 5.72  & 1.40  & 0.67  & 0.77  & 1.11  & 1.49  & 1.95  & 2.54  & 3.20 \\
     \#9    & 4.03  & 2.16  & 3.21  & 3.66  & 3.47  & 4.36  & 5.15  & 5.35  & 5.97 \\
     \#10    & 4.62  & 1.01  & 3.14  & 3.84  & 3.81  & 4.47  & 5.05  & 5.24  & 5.68 \\
              \hline
    Mean      & 7.73  & 6.09  & 4.18  & 4.14  & 3.98  & 4.47  & 4.95  & 5.12  & 5.62 \\
              \hline
    \end{tabular}%
  \caption{$\ell_2$ relative error of prediction of $U^{n+1}$ using $\mathcal{NN}(\cdot)$.}
    \label{tab:full_error}%
\end{table}%
In Table \ref{tab:full_error}, the columns shows the prediction error for $U^{n+1}, n=0,1,\cdots 8$ where the average error is computed for all time steps among testing samples which are less than $10\%$ in average.  Therefore, we claim that the proposed multi-layer reduced order neural network $\mathcal{NN}(\cdot)$ is effective in the aspect of prediction. We also claim that the coefficient $U^{n+1, W_n^2}$ for $n=0,1,\cdots, 8$ are sparse in the independent systems $W^{2}_n$. These systems are again learned simultaneously by training $\mathcal{NN}(\cdot)$.
%\newpage

%\subsubsection{Sparsity of $U^{n+1}$}
%
%\begin{figure}[htbp]
%%	\subfigure[U1 coordinate]{./fig/U1_coordinate.png}
%	\centering
%	\begin{subfigure}[b]{0.45\textwidth}
%		\includegraphics[width=\textwidth]{U1_coordinate}
%		\caption{$U^1$ in $W_1^2$-coordinate}
%		\label{fig:gull}
%	\end{subfigure}
%~
%\begin{subfigure}[b]{0.45\textwidth}
%	\includegraphics[width=\textwidth]{U9_coordinate}
%	\caption{$U^9$ in $W_9^2$-coordinate}
%	\label{fig:gull}
%\end{subfigure}
%~
%
%	\begin{subfigure}[b]{0.45\textwidth}
%		\includegraphics[width=\textwidth]{U1_average}
%		\caption{$U^1$ in $W_1^2$-coordinate quadratic average}
%		\label{fig:gull}
%	\end{subfigure}
%~
%	\begin{subfigure}[b]{0.45\textwidth}
%		\includegraphics[width=\textwidth]{U9_average}
%		\caption{$U^9$ in $W_9^2$-coordinate quadratic average}
%		\label{fig:gull}
%	\end{subfigure}
%\caption{Sparsity of output of $\mathcal{N}^i(U^0)$ in $W_i^2$, $i = 1,9$ }
%\end{figure}

\subsection{Clustering} \label{section_clustering}
In this experiment, we aim to model the fluid dynamics correspond to two different fractured media as shown in Figure \ref{fracture_network}. More specifically, the permeability coefficient of matrix region have $\kappa_m = 1$ and the permeability of the fractures are $\kappa_f = 10^3$. The one-step NLMC solutions pairs $(U^0, U^1)$ are generated following these two different configurations of fracture are then referred as ``Cluster 1'' and ``Cluster 2'' (see Figure \ref{fig:solutions} for an illustration). We will then compare the one-step prediction of networks $\mathcal {NN}_1 , \mathcal {NN}_2$ with that of $\mathcal{NN}_{\text{mixed}}$. The input are taken taken to be $U^0$, which are chosen to be the terminal solution of mobility driven 10-step nonlinear dynamics, while the output is an approximation of $U^1$.

$\mathcal {NN}_1 , \mathcal {NN}_2$ and $\mathcal{NN}_{\text{mixed}}$ share same one-layer soft thresholding neural network structure as in \eqref{eq:one_layer} while the first two network are trained with data for each cluster separately and the latter one is trained with mixed data for two clusters.

\begin{figure}[htbp]
	\centering
	\begin{subfigure}[b]{0.35\textwidth}
		\centering
		\includegraphics[width = \textwidth]{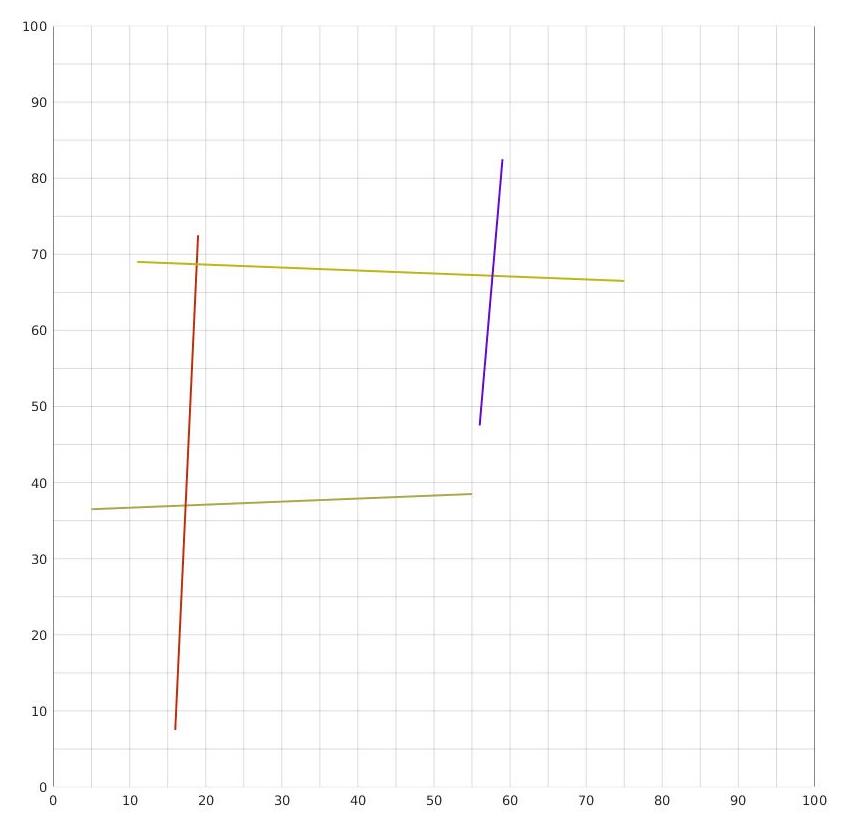}
		\caption{Cluster 1}
	\end{subfigure}%
	~
	\begin{subfigure}[b]{0.35\textwidth}
		\centering
		\includegraphics[width = \textwidth]{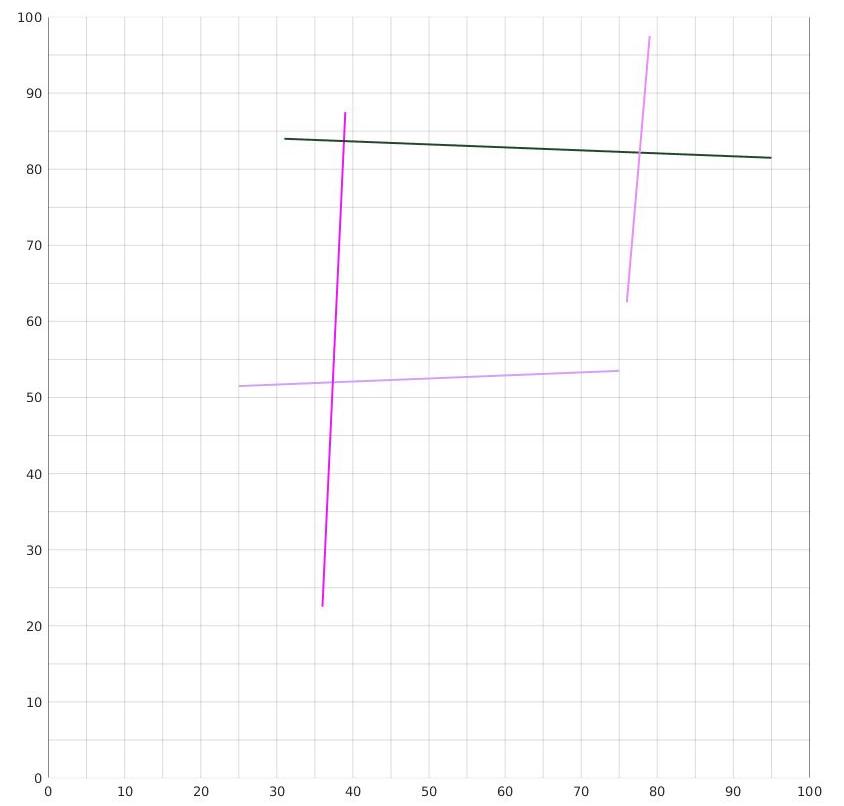}
		\caption{Cluster 2}
	\end{subfigure}
	\caption{Fracture networks for two clusters.}
	\label{fracture_network}
\end{figure}

\begin{figure}[htbp]
	\centering
	\begin{subfigure}[b]{0.35\textwidth}
		\centering
		\includegraphics[width = \textwidth]{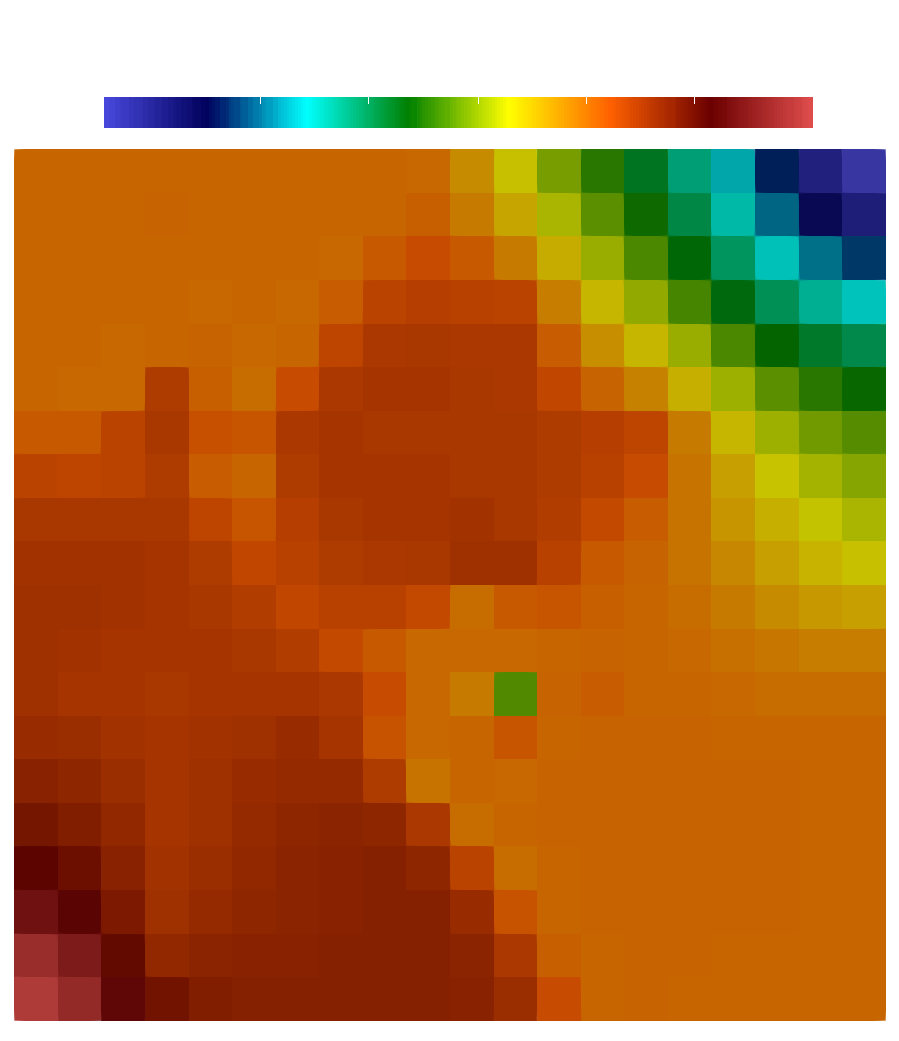}
		\caption{Coarse-scale NLMC solution $u^0$ -- Cluster 1}
	\end{subfigure}%
	~
	\begin{subfigure}[b]{0.35\textwidth}
		\centering
		\includegraphics[width = \textwidth]{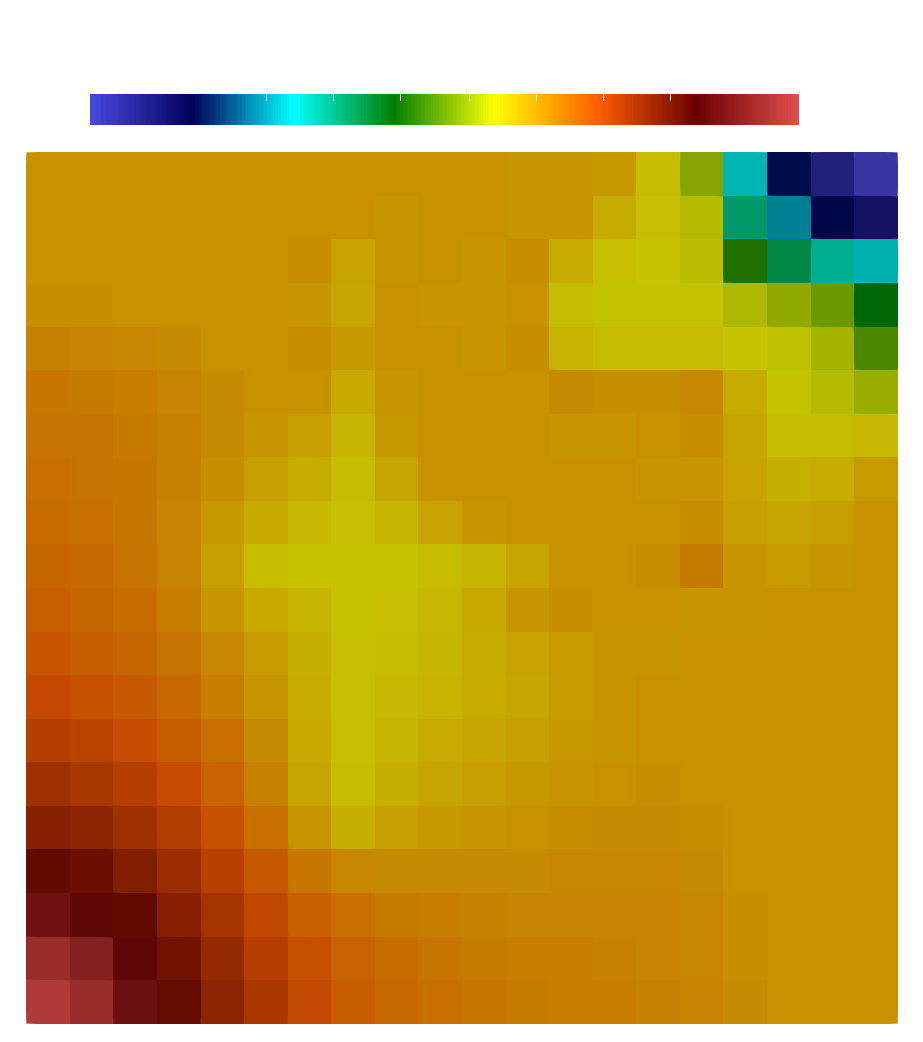}
		\caption{Coarse-scale solution of pressure $u^0$ -- Cluster 2}
	\end{subfigure}
	\begin{subfigure}[b]{0.35\textwidth}
		\centering
		\includegraphics[width = \textwidth]{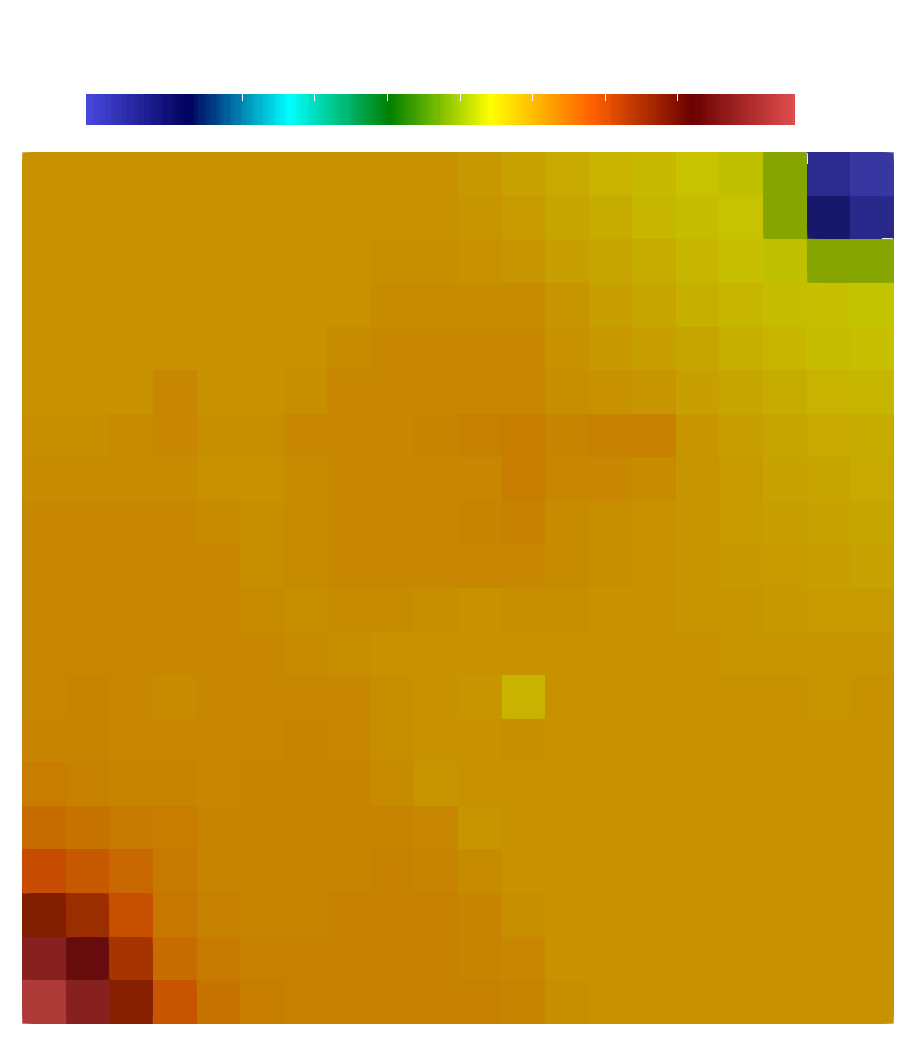}
		\caption{Coarse-scale solution of pressure $u^1$ -- Cluster 1}
	\end{subfigure}
	\begin{subfigure}[b]{0.35\textwidth}
		\centering
		\includegraphics[width = \textwidth]{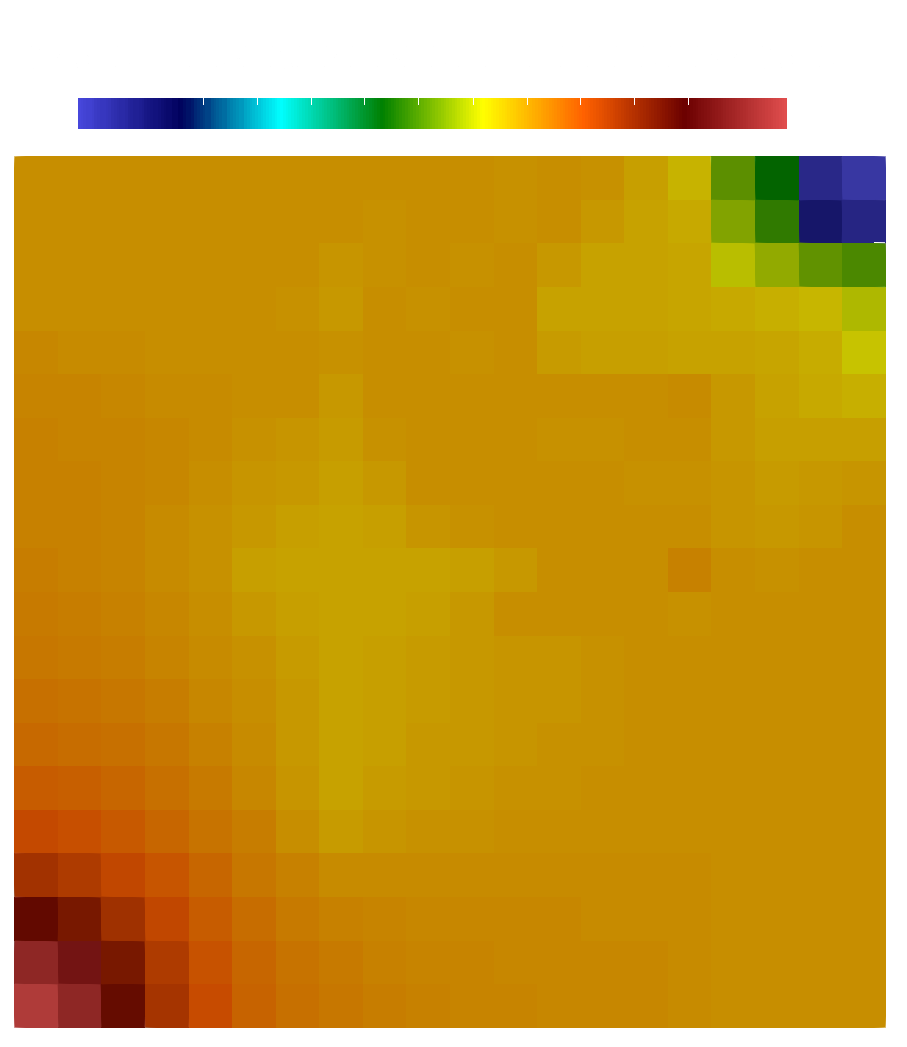}
		\caption{Coarse-scale solution of pressure $u^1$ -- Cluster 2}
	\end{subfigure}
	\caption{Example solution pairs for two different clusters.}
	\label{fig:solutions}
\end{figure}

Figure \ref{fracture_network} shows the fracture network we use to generate data for two clusters, while Figure \ref{fig:solutions} shows an example of solution $U^{0}$ and $U^1$ for each cluster. As observed from the profiles of $U^0$ and $U^1$, we can see that these the solutions are very different due to the translation of the fractures. Moreover, since the data resulted from both clusters have non-uniform map between $U^0$ and $U^1$, the mixed data set can be considered as obtained from a nonlinear map.
%\subsubsection{Training results}

%\begin{itemize}
%	\item training steps: 20000
%	\item $\gamma $ :0.001
%	\item learning rate :10.0
%\end{itemize}
% Table generated by Excel2LaTeX from sheet 'NLMC_gamma_001_error_table'

% Table generated by Excel2LaTeX from sheet 'NLMC_gamma_001_error_table'
\begin{table}[htbp]
	
	\centering
	\begin{subtable}[t]{0.45\textwidth}
		
	\begin{tabular}{|c|p{1.5cm}|p{1.5cm}|}
		\hline
		Sample Index &  Cluster 1 Error & Cluster 2 Error \\
		\hline
	\#1     & 0.87  & 0.26 \\
	\#2     & 3.19  & 0.45 \\
	\#3     & 13.62 & 0.79 \\
	\#4     & 12.36 & 0.49 \\
	\#5     & 7.64  & 0.43 \\
	\#6     & 10.69 & 1.00 \\
	\#7     & 19.08 & 0.53 \\
	\#8     & 2.57  & 0.92 \\
	\#9     & 0.77  & 0.46 \\
	\#10    & 3.70  & 0.38 \\
		\hline
		Mean  & 7.45  & 0.57 \\
		\hline
	\end{tabular}%
	\caption{Relative $\ell_2 $ prediction error(\%) of $\mathcal{N}_1$ and $\mathcal{N}_2$.}
	\label{tab:seperate training}%
\end{subtable}
% Table generated by Excel2LaTeX from sheet 'true_mixed_001_error_table'
	\begin{subtable}[t]{0.45\textwidth}
	\centering
	\begin{tabular}{|c|p{1.5cm}|p{1.5cm}|}
		\hline
		Sample Index&  Cluster 1 Error  & Cluster 2 Error \\
		\hline
	 \# 1     & 36.25 & 84.83 \\
	\#2     & 35.08 & 84.20 \\
	\#3     & 28.20 & 111.05 \\
	\#4     & 44.60 & 69.35 \\
	\#5     & 32.32 & 89.30 \\
	\#6     & 30.03 & 106.41 \\
	\#7     & 48.61 & 57.76 \\
	\#8     & 35.63 & 88.46 \\
	\#9     & 36.89 & 88.67 \\
	\#10    & 38.78 & 83.85 \\
			\hline
Mean	& 36.64 & 86.39 \\
		\hline
	\end{tabular}%
	\caption{Relative $\ell_2 $ prediction error(\%) of $\mathcal{N}_{\text{mixed}}$.}
	\label{tab:mixed_training}%
	\end{subtable}
\caption{Prediction error of $\mathcal{N}_1$, $\mathcal{N}_2$ and $\mathcal{N}_\text{mixed}$.}
\label{tab:cluster}
\end{table}%

Table \ref{tab:cluster} demonstrates the comparison of the prediction accuracy when the network is fed with a single cluster data and when the network is fed with a mixed data. This simple treatment can significantly improve the accuracy. For Cluster 1, the average prediction error of $\mathcal{NN}_1$ is around $7.45\%$ while that of the $\mathcal{NN}_{\text{mixed}}$ is around $36.64\%$. Similar contrast can also be observed for Cluster 2.

%\newpage
\section{Conclusion}\label{section_conclude}
In this paper, we discuss a novel deep 
neural network approach to 
model reduction approach for multiscale problems.  
To numerically solve multiscale problems,  
a fine grid needs to be used and results in a huge number of 
degrees of freedom.
To this end, non-local multicontinuum (NLMC) upscaling \cite{NLMC} is 
used as a dimensionality reduction technique. 
In flow dynamics problems, multiscale solutions at consecutive time instants are 
regarded as an input-output mechanism and learnt from deep neural networks techniques. 
By exploiting a relation between a soft-thresholding neural network and $\ell_1$ minimization, 
multiscale features of the coarse-grid solutions are extracted using neural networks. 
This results in a new neural-network-based construction of reduced-order model, 
which involves extracting appropriate important modes at each time step. 
We also suggest an efficient strategy for a class of nonlinear flow problems.
Finally, we present numerical examples to demonstrate the performance of 
our method applied to some selected problems.

%\newpage
\bibliographystyle{plain}
\bibliography{references,referenced_deep_MS}

\end{document}